\documentclass[a4paper]{elsarticle}

\usepackage{tikz-cd, tikz, graphicx}
\usepackage{lipsum}
\usepackage{amsmath,amsfonts,amsthm}
\usepackage{algorithmic}
\usepackage{algorithm}
\usepackage{array}
\usepackage[caption=false,font=normalsize,labelfont=sf,textfont=sf]{subfig}
\usepackage{textcomp}
\usepackage{stfloats}
\usepackage{url}
\usepackage{verbatim}
\usepackage{graphicx}
\usepackage{amssymb}
\usepackage{mwe}
\usepackage{booktabs}
\usetikzlibrary{shapes,decorations.pathreplacing,decorations.pathmorphing,calc}

\newtheorem{theorem}{Theorem}[section]
\newtheorem{lemma}[theorem]{Lemma}
\newtheorem{Corollary}[theorem]{Corollary}
\newtheorem{Proposition}[theorem]{Proposition}

\theoremstyle{definition}
\newtheorem{Remark}[theorem]{Remark}
\newtheorem{Example}[theorem]{Example}
\newtheorem{Definition}[theorem]{Definition}
\begin{document}
\begin{frontmatter}

\title{Commutative Distance-Regularity: Algebraic Hierarchies and Cayley Graph Constructions}
\address[AO]{Centro de Investigaciones en Ciencias Exactas y Naturales (CIEM), Universidad Nacional de Córdoba; Argentina.}
\address[AP1]{Instituto de Matem\'atica Aplicada San Luis, Universidad Nacional de San Luis and CONICET, San Luis, Argentina.}
\address[AP2]{Departamento de Matem\'atica, Universidad Nacional de San Luis, San Luis, Argentina.}

\author[AO]{Agust\'in Orom\'i}
\author[AP1,AP2]{Adri\'an Pastine}

\begin{abstract}
This paper resolves several open questions regarding the algebraic and metric inheritance of distance-based regularities under graph products, with a particular focus on Commutative Distance Degree-Regular (CDDR) graphs. We characterize the behavior of CDDR, distance mean-regular (DMR), and distance degree-regular (DDR) graphs under the strong and direct (tensor) products. To address the non-closure of classical regularities under the direct product, we introduce and study a novel, broader family of graphs: \textit{Distance Walk Regular (DWR)} graphs. Furthermore, we relax the global commutativity condition by defining and structuralizing the class of $i$-CDDR graphs, providing deep characterizations within the realm of Cayley graphs. By developing systematic lifting operations, we construct infinite families of graphs with arbitrarily large order and diameter within specific structural intersections, effectively settling topological gaps in the current regularity hierarchy.
\end{abstract}


\end{frontmatter}
\section{Introduction}
All graphs considered in this manuscript are assumed to be simple, finite, and connected. Given a graph $G$, we write $V(G)$ for its vertex set and $E(G)$ for its edge set. Two vertices $u, v \in V(G)$ are \emph{adjacent}, written $u \sim v$, whenever $uv \in E(G)$. The \emph{neighborhood} of a vertex $v \in V(G)$ is $\Gamma(v) = \{u \in V(G) : uv \in E(G)\}$. For $u, v \in V(G)$, the \emph{distance} $dist(u,v)$ is the length of a shortest $u$-$v$ path in $G$, and the \emph{diameter} $diam(G)$ is $\max_{u,v \in V(G)} dist(u,v)$.

Recall that the \emph{adjacency matrix} $A_G$ of $G$ has entries
\begin{equation*}
    (A_G)_{uv} =
    \begin{cases}
        1 & \text{if } uv \in E(G), \\
        0 & \text{otherwise,}
    \end{cases}
\end{equation*}
for $u, v \in V(G)$. More generally, for each integer $i \geq 0$, the \emph{$i$-th distance matrix} $A_i(G)$ is given by
\begin{equation*}
    (A_i(G))_{uv} =
    \begin{cases}
        1 & \text{if } dist(u, v) = i, \\
        0 & \text{otherwise.}
    \end{cases}
\end{equation*}
Equivalently, $A_i(G)$ is the adjacency matrix of the $i$-th distance graph $G_i$, in which two vertices are joined precisely when they lie at distance $i$ in $G$. In particular, $A_0(G) = I$ and $A_1(G) = A_G$, while $A_i(G) = 0$ for every $i > diam(G)$. The collection $\mathcal{A}_G = \{A_0(G), A_1(G), \dots, A_{diam(G)}(G)\}$ is a linearly independent set whose sum equals $J$, the all-ones matrix.

The spectrum of a graph $G$, denoted by $\operatorname{Spec}(G)$, is defined as the multiset of eigenvalues of its adjacency matrix $A_1(G)$. Since $A_1(G)$ is a real symmetric matrix, all its eigenvalues are real, and we denote them as $\lambda_1 \geq \lambda_2 \geq \dots \geq \lambda_n$. Frequently, we express the spectrum as $\operatorname{Spec}(G) = \{\lambda_1^{m_1}, \lambda_2^{m_2}, \dots, \lambda_k^{m_k}\}$, where the exponent $m_j$ denotes the algebraic multiplicity of the eigenvalue $\lambda_j$.

Furthermore, throughout this work, we denote by $G_i$ the $i$-th distance graph of $G$, for $1 \leq i \leq \operatorname{diam}(G)$. Formally, $G_i$ is defined as the graph with the same vertex set as $G$, where two vertices are adjacent if and only if their distance in $G$ is exactly $i$. Consequently, the adjacency matrix of $G_i$, denoted $A_1(G_i)$, is identically the $i$-th distance matrix of the original graph $G$, i.e., $A_1(G_i) = A_i(G)$. This notation allows us to treat the distance shells of $G$ as standard adjacency relationships, facilitating the spectral analysis of commutative distance-regularity.

Several well-known graph matrices---including the distance matrix, the Harary matrix, and the generalized adjacency-distance matrix of \cite{PasRoj2023}---arise as linear combinations of the elements of $\mathcal{A}_G$. When $G$ is regular, this linear-combination representation extends further, encompassing the Laplacian, normalized Laplacian, Randi\'c, and signless Laplacian matrices as well.

Let $G$ and $H$ be two graphs. The concept of graph products refers to various methods of combining two or more graphs to create a new graph. These products provide different structures and properties, which are highly useful for various applications in algebraic graph theory.

In this work, we study how regularity properties behave under graph products. We focus primarily on two fundamental operations:

The \textit{Strong Product} of two Graphs G and H, is a graph defined as follows: The vertex set of the graph $G\boxtimes H$ is the Cartesian product of the vertex sets of $G$ and $H$. Thus, the vertices are ordered pairs $(g, h)$, where $g \in V(G)$ and
$h \in V(H)$. Furthermore, two vertices $(g_1, h_1)$ and $(g_2, h_2)$ are adjacent in $G\boxtimes H$ if and only if one of the following conditions
is satisfied: $g_1 \sim g_2$ in $G$ and $h_1 = h_2$ , $g_1 = g_2$ and $h_1 \sim h_2$ in $H$ or  $g_1 \sim g_2$ in $G$ and $h_1 \sim h_2$ in $H$.
    
The \textit{Direct Product} (or Tensor Product) of two graphs $G$ and $H$, denoted as $G \times H$, is a new graph constructed as follows: the vertices
of the graph $G\times H$ are the pairs $(g, h)$, where $g \in  V(G)$ and $h \in V(H)$. Two vertices $(g_1, h_1)$ and $(g_2, h_2)$ are connected by
an edge in $G\times H$ if  $g_1 \sim g_2$ in $G$ and $h_1 \sim h_2$ in $H$.

The \textit{Cartesian product} $G \square H$ has vertex set $V(G) \times V(H)$, and two vertices $(g_1, h_1)$ and $(g_2, h_2)$ are adjacent whenever $g_1 = g_2$ and $h_1 \sim h_2$ in $H$, or $h_1 = h_2$ and $g_1 \sim g_2$ in $G$. This product has been extensively employed by Conde et al.\ \cite{conde2026}.

The Lexicographic Product of two graphs $G$ and $H$, denoted as $G[H]$, has the vertex set $V(G) \times V(H)$. Two vertices $(g_1, h_1)$ and $(g_2, h_2)$ are adjacent in $G[H]$ if and only if either $g_1 \sim g_2$ in $G$, or $g_1 = g_2$ and $h_1 \sim h_2$ in $H$. The distance and spectral properties of this product have been recently explored by Conde et al. \cite{conde2026}.

For a comprehensive understanding of the properties that are preserved when performing the product of two graphs, we recommend the reader to refer to \cite{HaImKl2011}.

A complete graph, $K_n$, is the graph such that every pair of vertices is adjacent. The cycle $C_n$ is the graph with vertex set $\{v_1, v_2, \dots, v_n\}$, where $v_i \sim v_{i+1}$ for $i = 1, \dots, n-1$ and $v_n \sim v_1$. 

Let $H$ be a finite group and let $S \subseteq H$ be a generating set such that the identity $e \notin S$ and $S = S^{-1}$. The \textit{Cayley graph} $Cay(H, S)$ is defined as the graph with vertex set $H$, where two vertices $x, y \in H$ are adjacent if and only if $yx^{-1} \in S$. A graph is called \textit{circulant} if it is a Cayley graph over a cyclic group. Equivalently, $G$ is circulant if there exists an automorphism of $G$ that acts as a cyclic permutation of all its vertices, or if its adjacency matrix $A_G$ is a circulant matrix. Cycles and complete graphs are classical examples of circulant graphs.

Finally, we recall the definition of the Generalized Petersen graph. For integers $n \geq 3$ and $1 \leq k < n/2$, the \textit{Generalized Petersen graph} $GP(n,k)$ has vertex set $\{u_0, \dots, u_{n-1}\} \cup \{v_0, \dots, v_{n-1}\}$ and edge set consisting of $u_i \sim u_{i+1}$, $u_i \sim v_i$, and $v_i \sim v_{i+k}$ for all $0 \leq i \leq n-1$ (with indices taken modulo $n$).

Two $n \times n$ matrices $A$ and $B$ are \emph{simultaneously diagonalizable} if a single invertible matrix $S$ exists such that $S^{-1}AS$ and $S^{-1}BS$ are both diagonal. The following classical result characterizes this situation in terms of commutativity.

\begin{lemma}\cite[5F]{Strang1988}
    Two diagonalizable matrices $A$ and $B$ share a common eigenvector basis if and only if $AB = BA$.
\end{lemma}

This work is heavily motivated by the open questions recently posed by Conde et al. \cite{conde2026}. In their concluding remarks, the authors highlighted that, from an algebraic perspective, it remained unclear whether the class of CDDR graphs is closed under standard graph operations, such as the strong product or the tensor (direct) product. Furthermore, they pointed out the potential for refining the current regularity hierarchy by identifying proper subclasses or superclasses of CDDR graphs involving relaxed commutativity conditions. 

To address these open problems, the remainder of this article is organized as follows. In Section 2, we introduce the necessary preliminaries and notations. In Section 3, we investigate the strong product of these graphs, determining which algebraic and metric regularities are preserved. In Section 4, we extend several key results from Conde et al. concerning the Cartesian product, deriving new sharpness bounds for spectral determination. In Section 5, we analyze the direct product; to resolve the non-closure of classical regularities under this operation, we introduce the novel family of Distance Walk-Regular (DWR) graphs and prove its structural properties. Finally, we explore the question of relaxed commutativity by defining $i$-CDDR graphs, and we extend these deep structural characterizations to the realm of Cayley graphs in the concluding sections.

\section{Preliminaries}

In order to set the theoretical foundation for this manuscript, we review key concepts and families of graphs that are central to our study. We focus specifically on graphs characterized by distance-based structural symmetries and algebraic properties of their distance matrices. Let us first define the standard terminology used throughout this text.

Consider a graph $G$ having a vertex set $V(G)$ and an edge set $E(G)$. For any integer $1 \leq i \leq \operatorname{diam}(G)$, let $\Gamma_i(u)$ represent the subset of vertices located at an exact distance $i$ from a given vertex $u$. We define the following notable graph classes:

\begin{enumerate}
    \item $G$ is classified as \textit{distance-regular} (DR) if the intersection size $|\Gamma_i(u) \cap \Gamma_j(v)|$ is completely determined by $i$, $j$, and the distance $d(u, v)$, remaining invariant regardless of the specific choice of vertices $u, v \in V(G)$ \cite{BrCoNe1989}.
    \item $G$ is known as \textit{distance degree-regular} (DDR) (or super-regular) if the size of the distance neighborhoods satisfies $|\Gamma_i(u)| = |\Gamma_i(v)|$ for every pair of vertices $u, v \in V(G)$ and all $1 \le i \le \operatorname{diam}(G)$ \cite{BlQuKe1981,HiNo1984,Weichsel1982}.
    \item $G$ is termed \textit{distance mean-regular} (DMR) if the average intersection size over the $h$-th neighborhood, given by
    $$\frac{1}{|\Gamma_h(u)|} \sum_{v \in \Gamma_h(u)} |\Gamma_i(u) \cap \Gamma_j(v)|,$$
    depends solely on the parameters $h$, $i$, and $j$. This average must be independent of the initial vertex $u \in V(G)$ for any $1 \le i, j \le \operatorname{diam}(G)$ \cite{HiNo1984}.
    \item $G$ belongs to the \textit{distance-polynomial} (DP) class if every distance matrix $A_i$ can be expressed as a polynomial evaluated at the adjacency matrix $A_1$, for $i = 1, \dots , \operatorname{diam}(G)$ \cite{Weichsel1982}.
    \item $G$ is defined as \textit{commutative distance degree-regular} (CDDR) if, for every pair of vertices $u, v \in V(G)$ and any distances $1 \leq i, j \leq \operatorname{diam}(G)$, the intersection sizes are symmetric:
    $|\Gamma_i(u) \cap \Gamma_j(v)| = |\Gamma_i(v) \cap \Gamma_j(u)|$ \cite{conde2026}.
\end{enumerate}

Let $\mathcal{A}_G=\{A_0,A_1,\dots,A_d\}$ denote the set of distance matrices of a CDDR graph $G$. Because the matrices in this family mutually commute, they admit a simultaneous diagonalization (as shown in \cite{conde2026}). Thus, there exists a non-singular matrix $S$ such that
$$S^{-1}A_iS=D_i \quad \text{for all } i\in\{0,1,\dots,d\},$$
where each $D_i$ is a diagonal matrix. From this point forward, we denote the diagonal entries of $D_i$ as $(D_i)_{jj}=\lambda_j^{(i)}$.

A comparative framework for several of these classes was provided by Diego and Fiol in \cite{DiFi2017}, where they established the strict inclusions:
\begin{equation*}
    \{\text{DR}\} \subsetneq \{\text{DMR}\} \subsetneq \{\text{DDR}\}.
\end{equation*}
Subsequently, Conde et al. \cite{conde2026} expanded this taxonomy by analyzing the commutativity property, leading to the hierarchy:
\begin{equation*}
    \{\text{DR}\} \subsetneq \{\text{DP}\} \subsetneq \{\text{CDDR}\} \subsetneq \{\text{DDR}\}.
\end{equation*}
They additionally proved that the classes $\{\text{DMR}\}$ and $\{\text{CDDR}\}$ intersect non-trivially but are not strictly contained within one another, highlighting a symmetric difference between the two sets.

To further analyze DMR graphs, we rely on the algebraic characterizations introduced in \cite{DiFi2017}. We first adopt their notation: for any vertex $u \in V(G)$, let $\omega_{ij}(u)$ denote the number of edges connecting a vertex located at distance $i$ from $u$ to another vertex at distance $j$ from $u$. Formally:
\begin{equation*}
    \omega_{ij}(u) = |\{vw \in E(G) : d(u, v) =i, d(u, w) = j\}|.
\end{equation*}

\begin{Proposition}\cite[Proposition~4.2]{DiFi2017}\label{lem: caracterización de DMR}
    Let $G$ be a DDR graph. Then $G$ is a DMR graph if and only if the value of $\omega_{ii}(u)$ is a constant with respect to $u$, for every $i = 1, \dots, \operatorname{diam}(G)$ and $u \in V(G)$.
\end{Proposition}

\begin{lemma}\cite[Proposition 4.5]{DiFi2017}\label{lem: caracterización de DMR con matrices}
    A graph $G$ with diameter $d$ and distance matrices $A_0, A_1, \dots, A_d$ is DMR if and only if, for all $h,i,j = 0, \dots, d$, the all-ones vector $\mathbf{1}$ is an eigenvector of the matrix $A_i A_j \circ A_h$, where $\circ$ denotes the Hadamard product.
\end{lemma}

Moreover, Conde et al. \cite{conde2026} deduced the following core spectral and structural criteria for CDDR graphs, which will be extensively applied in our proofs:

\begin{theorem}\cite[Theorem 3]{conde2026}
    Let $G$ be a graph with diameter $d$. Then, $G$ is a CDDR graph if and only if $A_i(G)A_j(G)=A_j(G)A_i(G)$ for all $1\le i, j \le d$.
\end{theorem}

\begin{theorem}\cite[Theorem 20]{conde2026}\label{teo: caracterización de CDDR 2}
    Let $G$ be a graph with diameter $d$. Then, $G$ is a CDDR graph if and only if its distance matrices $A_1(G), \dots , A_d(G)$ are simultaneously diagonalizable.
\end{theorem}

\begin{Proposition}\cite[Proposition 22]{conde2026}\label{prop: caracterizacion de DP}
    Let $G$ be a CDDR graph. Then, the following statements are equivalent:
    \begin{enumerate}
        \item $G$ is a DP graph. 
        \item Given $1 \le j, k \le \operatorname{diam}(G)$, if $(D_1(G))_{jj} = (D_1(G))_{kk}$, then $(D_i(G))_{jj} = (D_i(G))_{kk}$ for all $1 \le i \le \operatorname{diam}(G)$. 
    \end{enumerate}
\end{Proposition}

Analyzing these regular graph structures naturally links to the algebraic combinatorics of association schemes, as these schemes expertly handle sets of binary matrices encoding distance relationships. We briefly outline the formal definition of an association scheme for completeness. A $d$-class association scheme is a set $\mathcal{A} = \{A_0, \dots, A_d\}$ of binary matrices satisfying the following axioms for all $1 \le i, j \le d$:
\begin{enumerate}
    \item[(a)] $A_0 = I$,
    \item[(b)] $\sum_{i=0}^{d} A_i = J$,
    \item[(c)] $A_i^T \in \mathcal{A}$,
    \item[(d)] $A_i A_j = A_j A_i$,
    \item[(e)] $A_j A_i \in \operatorname{span}(\mathcal{A})$.
\end{enumerate}

This axiomatic definition implies several structural consequences. For instance, property (b) not only establishes the linear independence of the matrices $A_0, \dots, A_d$ but also forces constant row and column sums for each $A_i$. Property (e) dictates that the linear span of $\mathcal{A}$ forms an algebra of dimension $d + 1$. If all matrices in $\mathcal{A}$ are symmetric, the scheme itself is called symmetric, which renders condition (d) automatically satisfied. Comprehensive treatments of this topic can be found in \cite{BrCoNe1989, GoMe2016}.

Consolidating these relationships, Conde et al. \cite{conde2026} constructed a diagram mapping the intersections and inclusions of these regular families within the overarching DDR class (see Fig. \ref{fig: comparacion}).

To visually differentiate the status of known graph families within this diagram, two specific symbols are deployed. Black squares denote isolated examples or baseline graphs previously identified in the literature. Conversely, red circles signify the existence of infinite families of graphs occupying that specific structural niche. A central contribution of the present paper is to upgrade these black squares into red circles by introducing systematic lifting operators that construct such infinite regular families.

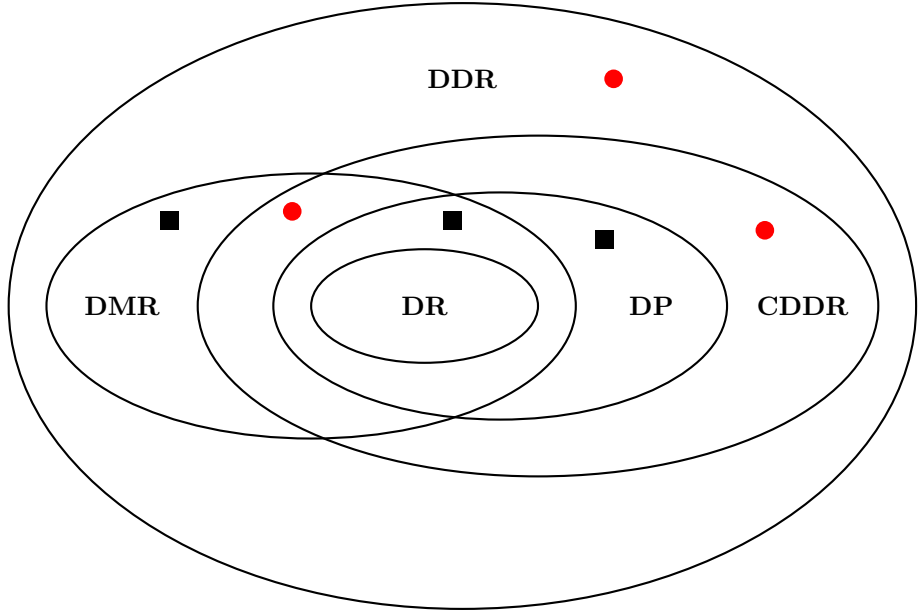
\begin{figure}
\begin{center}
    \begin{tikzpicture}
        \draw[thick] (-0.5,0) ellipse (1.5cm and 0.75cm);
        \node[] at (-0.5,0) {\textbf{DR}};
        \draw[thick] (1,0) ellipse (4.5cm and 2.25cm);
        \node[] at (4.5,0) {\textbf{CDDR}};
        \fill[red] (4,1) circle(3.5pt);
        \node[below] at (4,1) {};
        \draw[thick] (0.5,0) ellipse (3cm and 1.5cm);
        \node[] at (2.5,0) {\textbf{DP}};
        \fill[thick] (1.75,.75) rectangle ++(7pt,7pt);;
        \node[below] at (2,1) {};
        \fill (-.25,1) rectangle ++(7pt,7pt);;
        \node[below] at (0,1.25) {};
        \draw[thick] (-2,0) ellipse (3.5cm and 1.75cm);
        \node[] at (-4.5,0) {\textbf{DMR}};
        \fill (-4,1) rectangle ++(7pt,7pt);;
        \node[below] at (-4,1) {};
        \fill[red] (-2.25,1.25) circle(3.5pt);
        \node[below] at (-2.25,1.25) {};
        \draw[thick] (0,0) ellipse (6cm and 4cm);
        \node[] at (0,3) {\textbf{DDR}};
        \fill[red] (2,3) circle(3.5pt);
        \node[below] at (2,3) {};
    \end{tikzpicture}
\end{center}
\caption{The inclusion relationships and structural hierarchy among the families of distance-regular graphs studied by Conde et al. \cite{conde2026}.}\label{fig: comparacion}
\end{figure}

\section{Strong Product}
In this section, we analyze the structural and algebraic implications of the strong product operation, $\boxtimes$, focusing on the inheritance of commutative distance degree-regularity (CDDR). Given that CDDR graphs are characterized by a natural symmetry in the distribution of vertices at fixed distances—namely, satisfying the condition $|\Gamma_i(u) \cap \Gamma_j(v)| = |\Gamma_i(v) \cap \Gamma_j(u)|$ for all vertices and distance levels—it becomes crucial to understand how this internal metric balance behaves under topological expansions. 

To analyze the distance matrices of $G \boxtimes H$, it is convenient to first formalize the algebraic space in which they operate. We do this by introducing the tensor product algebra generated by the distance-matrix bases of the factors.

\begin{Definition}
    Let $G$ and $H$ be graphs with diameters $d_G$ and $d_H$, and let $\mathcal{A}_G$ and $\mathcal{A}_H$ be their respective distance-matrix bases. The tensor product algebra, denoted by $\mathcal{A}_G \otimes \mathcal{A}_H$, is defined as the subspace spanned by all Kronecker products of the form $A \otimes B$, with $A \in \operatorname{span}(\mathcal{A}_G)$ and $B \in \operatorname{span}(\mathcal{A}_H)$. 
\end{Definition}

With this algebraic framework established, we now determine the precise relationship between the distance matrices of the factor graphs and those of their strong product. By virtue of the supremum metric governing distances in $G \boxtimes H$, we can express the distance matrices of the product graph explicitly as elements of $\mathcal{A}_G \otimes \mathcal{A}_H$. This fundamental connection, which extends the formula for the adjacency matrix of the strong product to arbitrary distances, is formalized in the following lemma.

\begin{lemma}\label{lema: formula producto fuerte}
  Let $G$ and $H$ be two graphs. The distance matrix of order $i$ of the strong product $G \boxtimes H$ is given by.
    \begin{equation*}
        A_i(G\boxtimes H)= A_i(G) \otimes \left(\sum_{k=0}^i A_k(H) \right) + \left(\sum_{j=0}^{i-1} A_j(G) \right) \otimes A_i(H).
    \end{equation*}
    In particular $A_i(G \boxtimes H) \in span(\mathcal{A}_G \otimes \mathcal{A}_H)$
\end{lemma}

This is a matrix reformulation of Theorem 2.3 in \cite{Bresar2019}, which characterizes the exact distance-p graph of the strong product in combinatorial terms. The Kronecker product formulation given here is the form required for the spectral and algebraic arguments of Section 3.

\begin{lemma} \label{lema: AoB esta en el span}
    Let $G$ be a graph with diameter $d$, and let $\mathcal{A} = \{A_0, \dots, A_d\}$. 
    If $A, B \in \operatorname{span}(\mathcal{A})$, then $A \circ B \in \operatorname{span}(\mathcal{A})$.
\end{lemma}

\begin{proof}
    Let $A = \sum_{k=0}^d a_k A_k$ and $B = \sum_{k=0}^d b_k A_k$ be matrices in $\text{span}(\mathcal{A})$. Since the distance matrices of a graph have disjoint supports, their Hadamard product satisfies $A_k \circ A_j = \delta_{kj} A_k$. Therefore, we immediately obtain
    \begin{equation*}
        A \circ B = \sum_{k=0}^d a_k b_k A_k,
    \end{equation*}
    which clearly belongs to $\text{span}(\mathcal{A})$.
\end{proof}

With this algebraic framework in place, and using the mixed-product properties $(A \otimes B)(C \otimes D) = (AC) \otimes (BD)$ and $(A \otimes B) \circ (C \otimes D) = (A \circ C) \otimes (B \circ D)$, we are now ready to establish the closure of our regularity classes under the strong product.

\begin{theorem}\label{teo: clausura fuerte}
    The classes DDR, CDDR and DMR are closed under strong product.
\end{theorem}
\begin{proof}
    Let, $G$ and $H$ DDR, is easy using Lemma \ref{lema: formula producto fuerte} that $\mathbf{1}_{|V(G)|} \otimes \mathbf{1}_{|V(H)|}$ is an eigenvector of $A_i(G \boxtimes H)$ for every $1 \leq i \leq \max\{\operatorname{diam}(G),\operatorname{diam}(H)\}$, so $G \boxtimes H$ is DDR.
    
   If $G$ and $H$ are CDDR graphs, their distance-matrix spans, $\operatorname{span}(\mathcal{A}_G)$ and $\operatorname{span}(\mathcal{A}_H)$, are commutative algebras. It is a well-known property that the tensor product of two commutative algebras, in this case $\operatorname{span}(\mathcal{A}_G) \otimes \operatorname{span}(\mathcal{A}_H)$, is also a commutative algebra. Since Lemma \ref{lema: formula producto fuerte} guarantees that $A_i(G \boxtimes H) \in \operatorname{span}(\mathcal{A}_G) \otimes \operatorname{span}(\mathcal{A}_H)$, it follows immediately that $G \boxtimes H$ is CDDR.

    Finally, assume $G$ and $H$ are DMR graphs. Using the mixed-product properties of the Kronecker product (for both ordinary matrix multiplication and the Hadamard product), any expression of the form $(A_i(G \boxtimes H) A_j(G \boxtimes H) \circ A_h(G \boxtimes H)) \mathbf{1}_{|V(G \boxtimes H)|}$ can be linearly expanded into terms of the form
    \begin{equation*}
    ((A_{i'} A_{j'} \circ A_{h'}) \mathbf{1}_{|V(G)|}) \otimes ((A_{i''} A_{j''} \circ A_{h''}) \mathbf{1}_{|V(H)|}).
    \end{equation*}
Because $G$ and $H$ are DMR, each factor evaluates to a constant multiple of its respective all-ones vector. Consequently, the entire linear combination reduces to a scalar multiple of $\mathbf{1}_{|V(G \boxtimes H)|}$, proving that $G \boxtimes H$ is DMR.
\end{proof}

\begin{Corollary}\label{coro: autovalores de producto fuerte}
    Let $G$ and $H$ be CDDR graphs such that $S_G$ and $S_H$ are the matrices that simultaneously diagonalize $A_i(G)$ for all $1 \leq i \leq \operatorname{diam}(G)$ and $A_j(H)$ for all $1 \leq j \leq \operatorname{diam}(H)$. Then, $S_G \otimes S_H$ simultaneously diagonalizes $A_k(G \boxtimes H)$ for $1 \leq k \leq \max\{\operatorname{diam}(G), \operatorname{diam}(H)\}$. 

    Moreover, let $\lambda_j^{(i)}$ be the eigenvalue of $A_i(G)$ associated with the eigenvector $\mathbf{x}_j$, and let $\mu_m^{(k)}$ be the eigenvalue of $A_k(H)$ associated with the eigenvector $\mathbf{y}_m$. Then, the eigenvalue $\theta_{j,m}^{(k)}$ of the distance matrix $A_k(G \boxtimes H)$ associated with the eigenvector $\mathbf{x}_j \otimes \mathbf{y}_m$ is given by:
    \begin{equation*}
        \theta_{j,m}^{(k)} = \lambda_j^{(k)} \left( \sum_{l=0}^k \mu_m^{(l)} \right) + \left( \sum_{t=0}^{k-1} \lambda_j^{(t)} \right) \mu_m^{(k)}.
    \end{equation*} 
\end{Corollary}

\begin{proof}
    By Lemma \ref{lema: formula producto fuerte}, the distance matrix $A_k(G \boxtimes H)$ is expressed as a linear combination of Kronecker products of the distance matrices of $G$ and $H$. Since $\mathbf{x}_j$ and $\mathbf{y}_m$ are simultaneous eigenvectors of $\{A_i(G)\}$ and $\{A_j(H)\}$ respectively, standard properties of the Kronecker product imply that their tensor product $\mathbf{x}_j \otimes \mathbf{y}_m$ is an eigenvector of $A_k(G \boxtimes H)$. Evaluating $A_k(G \boxtimes H) (\mathbf{x}_j \otimes \mathbf{y}_m)$ directly yields the formula for the eigenvalue $\theta_{j,m}^{(k)}$. Since this holds for a complete basis of eigenvectors, $S_G \otimes S_H$ simultaneously diagonalizes all distance matrices of the strong product.
\end{proof}

As an immediate consequence of the corollary regarding the eigenvalue distribution of CDDR graphs, we can explicitly map the spectrum of their distance relations to establish spectral determination. This spectral rigidity leads directly to the following result.

\begin{Proposition}\label{teo: condicion suficiente para DP}
    Let $G$ and $H$ be DP graphs. If the map $\Phi: \operatorname{Spec}(G) \times \operatorname{Spec}(H) \to \mathbb{R}$ given by
    \begin{equation*}
        \Phi(\lambda, \mu) = \lambda \mu + \lambda + \mu
    \end{equation*}
    is injective, then $G \boxtimes H$ is DP.
\end{Proposition}

\begin{proof}
    It is a standard property of the strong product that the eigenvalues of $A_1(G \boxtimes H)$ are precisely of the form $\theta = \Phi(\lambda, \mu) = \lambda \mu + \lambda + \mu$, where $\lambda \in \operatorname{Spec}(G)$ and $\mu \in \operatorname{Spec}(H)$. Since $\Phi$ is injective, there are no spectral collisions; that is, $\Phi(\lambda, \mu) = \Phi(\lambda', \mu')$ if and only if $\lambda = \lambda'$ and $\mu = \mu'$. This injectivity ensures that the eigenspaces of $A_1(G \boxtimes H)$ do not degenerate and are spanned exactly by the tensor products of the respective eigenspaces of $G$ and $H$. By Corollary \ref{coro: autovalores de producto fuerte} and Proposition \ref{prop: caracterizacion de DP}, this spectral separation guarantees that $G \boxtimes H$ is a DP graph.
\end{proof}

\begin{Corollary}
    Let $G$ be a DP graph such that $-1 \notin \operatorname{Spec}(G)$. Then, the strong product $G \boxtimes K_n$ is a DP graph for every $n \geq 2$.
\end{Corollary}

\begin{proof}
    Recall that for $n \geq 2$, $\operatorname{Spec}(K_n) = \{ (n-1)^1, -1^{n-1} \}$. Suppose that the map $\Phi$ defined in Theorem \ref{teo: condicion suficiente para DP} is not injective. Then, there exist eigenvalues $\lambda_i, \lambda_j \in \operatorname{Spec}(G)$ such that
    \begin{equation*}
        \Phi(\lambda_i, n-1) = \Phi(\lambda_j, -1).
    \end{equation*}
    Expanding both sides of the equation yields:
    \begin{align*}
        \lambda_i(n-1) + \lambda_i + (n-1) &= \lambda_j(-1) + \lambda_j + (-1) \\
        n\lambda_i + n - 1 &= -1 \\
        n(\lambda_i + 1) &= 0.
    \end{align*}
    Since $n \geq 2$, this implies that $\lambda_i = -1$, which directly contradicts the hypothesis that $-1 \notin \operatorname{Spec}(G)$. Therefore, $\Phi$ must be injective. By Theorem \ref{teo: condicion suficiente para DP}, $G \boxtimes K_n$ is a DP graph.
\end{proof}

Having established that the classes of CDDR, DDR, and DMR graphs are closed under the strong product, and having provided sufficient conditions for the preservation of the DP property, a natural question arises: if the strong product $G \boxtimes H$ of two graphs $G$ and $H$ belongs to one of these regularity classes, must its factors necessarily belong to the same class?

The subsequent results address this question, detailing the specific conditions under which these regularities are strictly inherited by the factor graphs.

\begin{theorem}\label{teo: reciproco ddr}
    If the strong product $G \boxtimes H$ is a DDR graph, then both $G$ and $H$ are DDR graphs.
\end{theorem}

\begin{proof}
    Let $u, v \in V(G)$ and $x \in V(H)$. We aim to prove that $|\Gamma_i(u)| = |\Gamma_i(v)|$ for all $i$ by induction. 
    
    We begin with the base case $i=1$. Since $G \boxtimes H$ is DDR, we have $|\Gamma_1(u,x)| = |\Gamma_1(v,x)|$. A vertex $(w,z) \in \Gamma_1(u,x)$ if and only if one of the following disjoint cases holds:
    \begin{itemize}
        \item $u=w$ and $x \sim z$; there are exactly $|\Gamma_1(x)|$ such vertices.
        \item $u \sim w$ and $x=z$; there are exactly $|\Gamma_1(u)|$ such vertices.
        \item $u \sim w$ and $x \sim z$; there are exactly $|\Gamma_1(u)||\Gamma_1(x)|$ such vertices.
    \end{itemize}
    Thus, we can express the degree as $|\Gamma_1(u,x)| = |\Gamma_1(x)| + |\Gamma_1(u)| + |\Gamma_1(u)||\Gamma_1(x)|$. By an analogous argument for $|\Gamma_1(v,x)|$ and substituting into the initial equality, we obtain:
    \begin{align*}
        |\Gamma_1(x)| + |\Gamma_1(u)|(1 + |\Gamma_1(x)|) &= |\Gamma_1(x)| + |\Gamma_1(v)|(1 + |\Gamma_1(x)|) \\ 
        |\Gamma_1(u)| &= |\Gamma_1(v)|.
    \end{align*}
    This establishes the base case. Now, assume the result holds for all distances strictly less than $i$. To show $|\Gamma_i(u)| = |\Gamma_i(v)|$, we use the DDR property: $|\Gamma_i(u,x)| = |\Gamma_i(v,x)|$. The condition $d_{G \boxtimes H}((u,x),(w,z)) = i$ implies exactly one of the following occurs:
    \begin{itemize}
        \item $d_G(u,w) < i$ and $d_H(x,z) = i$; accounting for $|B_i(u)||\Gamma_i(x)|$ vertices.
        \item $d_G(u,w) = i$ and $d_H(x,z) < i$; accounting for $|\Gamma_i(u)||B_i(x)|$ vertices.
        \item $d_G(u,w) = i$ and $d_H(x,z) = i$; accounting for $|\Gamma_i(u)||\Gamma_i(x)|$ vertices.
    \end{itemize}
    Thus, $|\Gamma_i(u,x)| = |B_i(u)||\Gamma_i(x)| + |\Gamma_i(u)||B_i(x)| + |\Gamma_i(u)||\Gamma_i(x)|$. By the inductive hypothesis, the size of the ball $|B_i(u)|$ is equal to $|B_i(v)|$. Substituting the analogous expansion for $|\Gamma_i(v,x)|$ and subtracting the equal terms $|B_i(u)||\Gamma_i(x)| = |B_i(v)||\Gamma_i(x)|$ from both sides, we get:
    \begin{equation*}
        |\Gamma_i(u)|(|B_i(x)| + |\Gamma_i(x)|) = |\Gamma_i(v)|(|B_i(x)| + |\Gamma_i(x)|).
    \end{equation*}
    Since $x \in B_i(x)$, we have $|B_i(x)| \geq 1$. Therefore, we can divide both sides by the non-zero factor $(|B_i(x)| + |\Gamma_i(x)|)$ to conclude that $|\Gamma_i(u)| = |\Gamma_i(v)|$. This implies $G$ is DDR. An analogous argument shows that $H$ is also DDR.
\end{proof}

\begin{theorem}\label{teo: reciproco cddr}
    If the strong product $G \boxtimes H$ is CDDR, then both $G$ and $H$ are CDDR graphs.
\end{theorem} 

\begin{proof}
Let $u, v \in V(G)$. We must show that $|\Gamma_i(u) \cap \Gamma_j(v)| = |\Gamma_j(u) \cap \Gamma_i(v)|$ for all $i$ and $j$. By symmetry, we may assume $i < j$. Since $G \boxtimes H$ is CDDR, for any $x \in V(H)$, we have:
\begin{equation*}
|\Gamma_i(u,x) \cap \Gamma_j(v,x)| = |\Gamma_j(u,x) \cap \Gamma_i(v,x)|.
\end{equation*}
For a vertex $(w,z) \in \Gamma_i(u,x) \cap \Gamma_j(v,x)$, the condition $(w,z) \in \Gamma_j(v,x)$ with $d_H(x,z) \leq i < j$ implies $d_G(v,w) = j$ and $d_H(x,z) < j$. Evaluating the conditions for membership in $\Gamma_i(u,x)$, exactly one of the following disjoint cases occurs:
\begin{itemize}
\item $d_G(u,w) < i$ and $d_H(x,z) = i$; yielding $|B_i(u) \cap \Gamma_j(v)| |\Gamma_i(x)|$ vertices.
\item $d_G(u,w) = i$ and $d_H(x,z) < i$; yielding $|\Gamma_i(u) \cap \Gamma_j(v)| |B_i(x)|$ vertices.
\item $d_G(u,w) = i$ and $d_H(x,z) = i$; yielding $|\Gamma_i(u) \cap \Gamma_j(v)| |\Gamma_i(x)|$ vertices.
\end{itemize}
Thus, the intersection size expands as:
\begin{equation*}
|\Gamma_i(u,x) \cap \Gamma_j(v,x)| = |\Gamma_i(x)||B_i(u) \cap \Gamma_j(v)| + (|B_i(x)| + |\Gamma_i(x)|)|\Gamma_i(u) \cap \Gamma_j(v)|.
\end{equation*}
Proceeding by induction on $i$ and employing an algebraic argument completely analogous to that of Theorem \ref{teo: reciproco ddr}, the terms corresponding to the balls $B_i$ cancel out precisely on both sides of the equation. This simplification directly yields $|\Gamma_i(u) \cap \Gamma_j(v)| = |\Gamma_j(u) \cap \Gamma_i(v)|$. Therefore, $G$ is CDDR, and by a symmetrical argument, $H$ is also CDDR.
\end{proof}

\begin{theorem}
    If the strong product $G \boxtimes H$ is a DR graph, then both $G$ and $H$ are DR graphs.
\end{theorem}

\begin{proof}
    Let $u, v, z, w \in V(G)$ such that $d_G(u, v) = d_G(z, w)$. To show that $G$ is distance-regular, we must verify that the intersection numbers $|\Gamma_i(u) \cap \Gamma_j(v)|$ depend only on the distance $d_G(u, v)$. We proceed by induction on $i$.
    
    Fix $x \in V(H)$. Since $d_{G \boxtimes H}((u, x), (v, x)) = d_{G \boxtimes H}((z, x), (w, x))$ and $G \boxtimes H$ is distance-regular, we have:
    \begin{equation*}
        |\Gamma_i(u, x) \cap \Gamma_j(v, x)| = |\Gamma_i(z, x) \cap \Gamma_j(w, x)| \quad \text{for all } i, j.
    \end{equation*}
    By analyzing the structural properties of the strong product exactly as in Theorem \ref{teo: reciproco cddr}, we can explicitly expand the distance shells to obtain:
    \begin{equation*}
        |\Gamma_i(u, x) \cap \Gamma_j(v, x)| = |\Gamma_i(x)||B_i(u) \cap \Gamma_j(v)| + (|B_i(x)| + |\Gamma_i(x)|)|\Gamma_i(u) \cap \Gamma_j(v)|.
    \end{equation*}
    Applying the same expansion to the pair $(z,w)$ and substituting both into the initial equality, we can use the inductive hypothesis to cancel the terms corresponding to the balls $B_i$ on both sides (since they depend only on distances strictly smaller than $i$). Isolating the remaining terms, it immediately follows that $|\Gamma_i(u) \cap \Gamma_j(v)| = |\Gamma_i(z) \cap \Gamma_j(w)|$. Thus, $G$ is distance-regular. By a symmetric argument, $H$ is also distance-regular.
\end{proof}

\begin{theorem}\label{teo: si alguno no es dp el producto no es dp}
    If the strong product $G \boxtimes H$ is a DP graph, then both $G$ and $H$ are DP graphs.
\end{theorem}

\begin{proof}
    Since every DP graph is also CDDR, the hypothesis that $G \boxtimes H$ is DP implies it is CDDR. By Theorem \ref{teo: reciproco cddr}, both $G$ and $H$ must therefore be CDDR. 
    
    Assume, for the sake of contradiction, that $H$ is not DP. By Proposition \ref{prop: caracterizacion de DP}, there exist indices $k$ and $j$ such that $\mu_j^{(1)} = \mu_k^{(1)}$ but $\mu_j^{(i)} \neq \mu_k^{(i)}$ for some $i > 1$. Let $i$ be the minimum index such that this inequality holds; that is, $\mu_j^{(l)} = \mu_k^{(l)}$ for all $0 \leq l < i$. 
    
    Let $\lambda_1^{(1)}$ be the eigenvalue associated with the all-ones vector $\mathbf{1}_{|V(G)|}$. This eigenvalue corresponds to the valency of $G$, and thus $\lambda_1^{(1)} > 0$. The eigenvalue of $A_1(G \boxtimes H)$ associated with the pair $(1, j)$ is given by $\theta_{(1,j)}^{(1)} = \lambda_1^{(1)} + \mu_j^{(1)} + \lambda_1^{(1)}\mu_j^{(1)}$. Similarly, $\theta_{(1,k)}^{(1)} = \lambda_1^{(1)} + \mu_k^{(1)} + \lambda_1^{(1)}\mu_k^{(1)}$. Since $\mu_j^{(1)} = \mu_k^{(1)}$, it follows that $\theta_{(1,j)}^{(1)} = \theta_{(1,k)}^{(1)}$. 
    
    Now, applying Corollary \ref{coro: autovalores de producto fuerte} to the $i$-th distance matrix and algebraically rearranging the terms, we have:
    \begin{align*}
        \theta_{(1,j)}^{(i)} &= \lambda_1^{(i)} \left(\sum_{s=0}^{i-1} \mu_j^{(s)}\right) + \left(\sum_{t=0}^{i} \lambda_1^{(t)}\right)\mu_j^{(i)}, \\
        \theta_{(1,k)}^{(i)} &= \lambda_1^{(i)} \left(\sum_{s=0}^{i-1} \mu_k^{(s)}\right) + \left(\sum_{t=0}^{i} \lambda_1^{(t)}\right)\mu_k^{(i)}.
    \end{align*}
    By the minimality of $i$, the sums $\sum_{s=0}^{i-1} \mu_j^{(s)}$ and $\sum_{s=0}^{i-1} \mu_k^{(s)}$ are strictly equal. Furthermore, since the eigenvalue $\lambda_1^{(t)}$ associated with the all-ones vector represents the number of vertices at distance $t$ in $G$, the coefficient $\sum_{t=0}^{i} \lambda_1^{(t)}$ is strictly positive (as it contains at least $\lambda_1^{(0)}=1$ and the valency $\lambda_1^{(1)} > 0$). Given that $\mu_j^{(i)} \neq \mu_k^{(i)}$ and this coefficient is non-zero, we conclude that $\theta_{(1,j)}^{(i)} \neq \theta_{(1,k)}^{(i)}$. 
    
    Therefore, $G \boxtimes H$ possesses two identical eigenvalues for the adjacency matrix that correspond to distinct eigenvalues in the $i$-th distance matrix. By Proposition \ref{prop: caracterizacion de DP}, this implies that $G \boxtimes H$ is not DP, which directly contradicts our initial hypothesis. Thus, $H$ must be a DP graph. By a symmetric argument, $G$ is also a DP graph.
\end{proof}

Having established the inheritance properties for the DR, DP, DDR, and CDDR classes under the strong product, we turn to distance mean-regularity DMR. Because evaluating the DMR condition in $G \boxtimes H$ requires tracking intricate combinatorial interactions across both factor graphs, we first isolate the core mechanics in a preliminary lemma. This result establishes the precise behavior of the edge-counting parameters $\omega_{ii}$ and the block-sum distributions under the supremum metric, thereby streamlining the proof of our main theorem.

\begin{lemma}\label{lema: lema dmr}
    Let $G$ and $H$ graphs and $F=G \boxtimes H$. Then
    \begin{equation*}
    \begin{split}
    \omega_{ii}^{F}(u,t) &= \omega_{ii}^G(u) \omega_{ii}^H(t) + |B_{i+1}(u)| \cdot \omega_{ii}^H(t) \\
    &\quad + |B_{i+1}(t)| \cdot \omega_{ii}^G(u) + \omega_{i,i-1}^G(u) \omega_{i,i-1}^H(t) \\
    &\quad + \left[ \omega_{ii}^H(t) + |\Gamma_i(t)| \right] \left( \sum_{k=0}^{i-1} \omega_{kk}^G(u) + \omega_{k,k+1}^G(u) \right) \\
    &\quad + \left[ \omega_{ii}^G(u) + |\Gamma_i(u)| \right] \left( \sum_{k=0}^{i-1} \omega_{kk}^H(t) + \omega_{k,k+1}^H(t) \right).
    \end{split}
    \end{equation*}
\end{lemma}

\begin{proof}
     Recall that by definition,
    \begin{equation*}
        \omega_{ii}^F(u,t) = |{(x,s)(y,r) \in E(F) : (x,s), (y,r) \in \Gamma_i(u,t)}|.
    \end{equation*}
    We analyze the possible edges in the strong product by partitioning them into three disjoint types according to the adjacency in $G$ and $H$.

    \textbf{Case 1:} $x \sim y$ and $s = r$. Consider the edges in the strong product where adjacency occurs only in the first factor $G$. For an edge $(x,s) \sim (y,s)$ to satisfy $(x,s), (y,s) \in \Gamma_i(u,t)$, we distinguish the following sub-cases:
    \begin{itemize}
    \item $d(x,u)=i$, $d(y,u)=i$, and $d(s,t)<i$. There are $|B_i(t)| \cdot \omega_{ii}^G(u)$ such cases.
    \item $d(x,u)<i$, $d(y,u)=i$, and $d(s,t)=i$. Since $x \sim y$, this forces $d(x,u)=i-1$. There are $|\Gamma_i(t)| \cdot \omega_{i,i-1}^G(u)$ such cases.
    \item $d(x,u)<i$, $d(y,u)<i$, and $d(s,t)=i$. This configuration yields a total count of cases given by $|\Gamma_i(t)| \cdot \left( \sum_{k=0}^{i-1} \omega_{kk}^G(u) + \omega_{k,k-1}^G(u) \right)$.
    \item $d(x,u)=i$, $d(y,u)=i$, and $d(s,t)=i$. There are $|\Gamma_i(t)| \cdot \omega_{ii}^G(u)$ cases.
    \end{itemize}
    By collecting these terms and using the fact that $|B_{i+1}(t)| = |B_i(t)| + |\Gamma_i(t)|$, the total count for this case simplifies to:
    \begin{equation*}|B_{i+1}(t)| \cdot \omega_{ii}^G(u) + |\Gamma_i(t)| \cdot \left( \sum_{k=0}^{i-1} \omega_{kk}^G(u) + \omega_{k,k+1}^G(u) \right).
    \end{equation*}

    \textbf{Case 2:} $x = y$ and $s \sim r$.
    This case is strictly analogous to Case 1, focusing on the edges where adjacency occurs only in the second factor $H$. By symmetry, the total count of such edges $(x,s) \sim (x,r)$ where both endpoints lie in $\Gamma_i(u,t)$ is:
    \begin{equation*}
    |B_{i+1}(u)| \cdot \omega_{ii}^H(t) + |\Gamma_i(u)| \cdot \left( \sum_{k=0}^{i-1} \omega_{kk}^H(t) + \omega_{k,k+1}^H(t) \right).
    \end{equation*}

    \textbf{Case 3:} $x \sim y$ and $s \sim r$. Lastly, we consider the ``diagonal'' edges of the strong product, where adjacency occurs in both factors simultaneously. For an edge $(x,s) \sim (y,r)$ that $(x,s), (y,r) \in \Gamma_i(u,t)$, we must satisfy one of the following combinatorial configurations:
    \begin{itemize}
    \item $d(x,u)=i$, $d(y,u)=i$ and $d(s,t)<i, d(r,t)<i$: There are $\omega_{ii}^G(u) \cdot \left( \sum_{k=0}^{i-1} \omega_{kk}^H(t)+\omega_{k,k-1}^H(t) \right)$ such cases.
    \item $d(x,u)<i, d(s,t)=i$ and $d(y,u)=i, d(r,t)<i$: Since $d(x,y)=1$ and $d(s,r)=1$, this forces $d(x,u)=i-1$ and $d(r,t)=i-1$, totaling $\omega_{i,i-1}^G(u) \cdot \omega_{i,i-1}^H(t)$ cases.
    \item $d(x,u)=i, d(s,t)=i$ and $d(y,u)=i, d(r,t)<i$: This accounts for $\omega_{ii}^G(u) \cdot \omega_{i,i-1}^H(t)$ cases.
    \item $d(x,u)<i, d(s,t)=i$ and $d(y,u)<i, d(r,t)=i$: This configuration gives $\left( \sum_{k=0}^{i-1} \omega_{kk}^G(u)+\omega_{k,k-1}^G(u) \right) \cdot \omega_{ii}^H(t)$ cases.
    \item $d(x,u)=i, d(s,t)=i$ and $d(y,u)<i, d(r,t)=i$: This accounts for $\omega_{i,i-1}^G(u) \cdot \omega_{ii}^H(t)$ cases.
    \item $d(x,u)=i, d(s,t)=i$ and $d(y,u)=i, d(r,t)=i$, giving $\omega_{ii}^G(u) \cdot \omega_{ii}^H(t)$ cases.
    \end{itemize}
    Summing the six sub-cases detailed above, we obtain:
    \begin{align*}
    &\omega_{ii}^G(u) \omega_{ii}^H(t) + \omega_{i,i-1}^G(u) \omega_{i,i-1}^H(t) + \omega_{ii}^H(t) \left( \sum_{k=0}^{i-1} \omega_{kk}^G(u) + \omega_{k,k-1}^G(u) \right) \\
    &+ \omega_{ii}^G(u) \left( \sum_{k=0}^{i-1} \omega_{kk}^H(t) + \omega_{k,k-1}^H(t) \right).
\end{align*}
    By summing the contributions from all three cases and simplifying the resulting expression, we obtain:
    \begin{equation*}
    \begin{split}
    \omega_{ii}^{F}(u,t) &= \omega_{ii}^G(u) \omega_{ii}^H(t) + |B_{i+1}(u)| \cdot \omega_{ii}^H(t) \\
    &\quad + |B_{i+1}(t)| \cdot \omega_{ii}^G(u) + \omega_{i,i-1}^G(u) \omega_{i,i-1}^H(t) \\
    &\quad + \left[ \omega_{ii}^H(t) + |\Gamma_i(t)| \right] \left( \sum_{k=0}^{i-1} \omega_{kk}^G(u) + \omega_{k,k+1}^G(u) \right) \\
    &\quad + \left[ \omega_{ii}^G(u) + |\Gamma_i(u)| \right] \left( \sum_{k=0}^{i-1} \omega_{kk}^H(t) + \omega_{k,k+1}^H(t) \right).
    \end{split}
    \end{equation*}
\end{proof}

From this lemma and the characterization of DMR graphs within the DDR family established by Diego and Fiol (Proposition~\ref{lem: caracterización de DMR}) we obtain the following result.

\begin{theorem}
    If the strong product $G\boxtimes H$ is a DMR graph, then both $G$ and $H$ are DMR graphs.
\end{theorem}

\begin{proof}
    We aim to show that the parameter $\omega_{ii}^G(u)$ is independent of the choice of $u \in V(G)$ for all $1 \leq i \leq \operatorname{diam}(G)$. We proceed by strong induction on $i$. By Theorem \ref{teo: reciproco ddr}, we already know that $G$ and $H$ are DDR, meaning their distance shells and ball sizes are vertex-independent.

    Assume that for all $j < i$, the parameters $\omega_{jj}^G$ are well-defined. By the standard degree-sum identity for distance layers, $k |\Gamma_j(u)| = \omega_{j,j-1}^G(u) + 2\omega_{j,j}^G(u) + \omega_{j,j+1}^G(u)$, where $k$ is the valency of $G$. Since $G$ is DDR and by the inductive hypothesis, it immediately follows that the crossing parameters $\omega_{j,j+1}^G$ are also well-defined for all $j < i$.

    Let $u, v \in V(G)$ be distinct vertices and fix $t \in V(H)$. Since $G \boxtimes H$ is DMR, $\omega_{ii}^F(u,t) = \omega_{ii}^F(v,t)$. Applying Lemma \ref{lema: lema dmr} to both pairs and canceling all terms that depend only on the fixed factor $H$, the DDR property of $G$, and the parameters $\omega_{jj}^G, \omega_{j,j+1}^G$ for $j < i$ (which are identical for $u$ and $v$ by the inductive hypothesis), the equation simplifies strictly to:
    \begin{equation*}
        \omega_{ii}^G(u) \left( \omega_{ii}^H(t) + |B_{i+1}(t)| + |\Gamma_i(t)| \right) = \omega_{ii}^G(v) \left( \omega_{ii}^H(t) + |B_{i+1}(t)| + |\Gamma_i(t)| \right).
    \end{equation*}
    Since $|B_{i+1}(t)| \geq 1$, the coefficient is strictly positive. Dividing by this non-zero factor yields $\omega_{ii}^G(u) = \omega_{ii}^G(v)$. This completes the induction, establishing that $G$ is DMR. By a symmetric argument, $H$ is also DMR.
\end{proof}

The structural theorems for the strong product $\boxtimes$ are powerful constructive tools. By proving that the properties of DDR, CDDR, and DMR are preserved, we move beyond isolated examples to systematically generate infinite families of graphs within these regularity classes.

As evidenced by Conde \cite{conde2026}, the Möbius ladder $M$ (Fig. \ref{fig: Mobius ladder}) is both DMR and CDDR but fails to be DP. By choosing $G$ to be any DMR and CDDR graph, Theorems \ref{teo: clausura fuerte}, and \ref{teo: si alguno no es dp el producto no es dp} imply that $M \boxtimes G$ retains the DMR and CDDR properties while failing to be DP.

\begin{figure}

\begin{center}
\begin{tikzpicture}[scale=1]

\coordinate (A1) at (-1.5,0);         
\coordinate (A2) at (-3, 2);
\coordinate (A3) at (-1.5, 4);

\coordinate (B1) at (1.5, 4);     
\coordinate (B2) at (3, 2);
\coordinate (B3) at (1.5, 0);

\coordinate (C1) at (-1, 1);    
\coordinate (C2) at (-1.5, 2);
\coordinate (C3) at (-1, 3);

\coordinate (D1) at (1, 3);
\coordinate (D2) at (1.5, 2);
\coordinate (D3) at (1, 1);

\draw[thick] (A1) -- (A2) -- (A3) -- (B1) -- (B2) -- (B3);
\draw[thick] (C1) -- (C2) -- (C3) -- (D1) -- (D2) -- (D3);
\draw[thick] (A1) -- (C1);
\draw[thick] (A2) -- (C2);
\draw[thick] (A3) -- (C3);
\draw[thick] (B1) -- (D1);
\draw[thick] (B2) -- (D2);
\draw[thick] (B3) -- (D3);
\draw[thick] (B3) -- (C1);
\draw[thick] (D3) -- (A1);
\foreach \point in {A1,A2,A3,B1,B2,B3,C1,C2,C3,D1,D2,D3} {
    \fill (\point) circle (3pt);
}

\end{tikzpicture}
\end{center}

\caption{M\"obius ladder $M$ with 12 vertices}\label{fig: Mobius ladder} 
\end{figure}
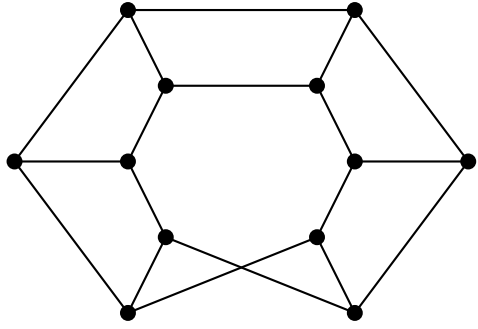

\begin{Corollary}
    For any $n, d \in \mathbb{N}$, there exists a graph $G$ with $|V(G)| \geq n$ and $\operatorname{diam}(G) \geq d$ such that $G$ is DMR and CDDR, but not DP.
\end{Corollary}

Similarly, utilizing the graph $X$ from \cite{conde2026} (Fig. \ref{fig: DMR y no CDDR}), which is DMR but not CDDR, the product $X \boxtimes G$ allows for the construction of an infinite family of DMR graphs that are not CDDR.

\begin{figure}

\begin{center}
\begin{tikzpicture}[scale=1]

\coordinate (A1) at (0,1);         
\coordinate (A2) at (1, -0.5);
\coordinate (A3) at (-1, -0.5);

\coordinate (B1) at (0, 3);     
\coordinate (B2) at (1.5, 2.5);
\coordinate (B3) at (-1.5, 2.5);

\coordinate (C1) at (2.75, -1.25);    
\coordinate (C2) at (3, 0);
\coordinate (C3) at (1.5, -2);

\coordinate (D1) at (-2.75, -1.25);
\coordinate (D2) at (-3, 0);
\coordinate (D3) at (-1.5, -2);

\draw[thick] (A1) -- (A2) -- (A3) -- (A1);
\draw[thick] (B1) -- (B2) -- (B3) -- (B1);
\draw[thick] (C1) -- (C2) -- (C3) -- (C1);
\draw[thick] (D1) -- (D2) -- (D3) -- (D1);
\draw[thick] (C2) -- (C1) -- (C3) -- (C2);
\draw[thick] (A1) -- (A2) -- (A3) -- (A1);
\draw[thick] (A1) -- (B1);
\draw[thick] (A2) -- (C1);
\draw[thick] (A3) -- (D1);
\draw[thick] (B2) -- (C2);
\draw[thick] (C3) -- (D3);
\draw[thick] (D2) -- (B3);
\foreach \point in {A1,A2,A3,B1,B2,B3,C1,C2,C3,D1,D2,D3} {
    \fill (\point) circle (3pt);
}

\end{tikzpicture}
\end{center}

\caption{A DMR graph which is not CDDR}\label{fig: DMR y no CDDR} 
\end{figure}
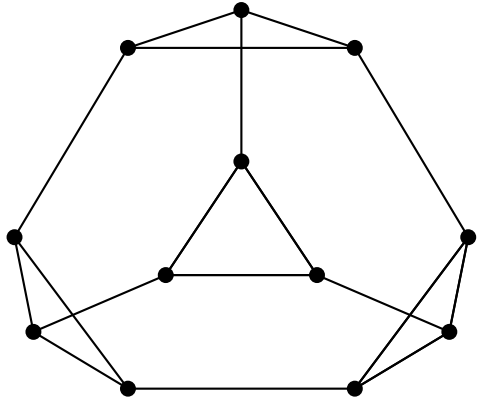

\begin{Corollary}
    For any $n, d \in \mathbb{N}$, there exists a graph $G$ with $|V(G)| \geq n$ and $\operatorname{diam}(G) \geq d$ such that $G$ is DMR but not CDDR.
\end{Corollary}

\begin{Remark}
    The eigenvalue $\lambda_i^{(1)} + \mu_j^{(1)} + \lambda_i^{(1)}\mu_j^{(1)}$ of $A_1(G \boxtimes H)$ implies that if $A_1(H)$ has an eigenvalue $-1$, then $-1$ is an eigenvalue of $A_1(G \boxtimes H)$ with multiplicity at least $|V(G)|$. In CDDR graphs, this spectral degeneracy restricts the number of distinct eigenvalues, providing a robust mechanism to construct graphs that are CDDR but not DP.
\end{Remark}

\begin{Corollary}
Let $G$ be a CDDR graph and let $\mathbf{x}$ be an eigenvector such that $A_1(G)\mathbf{x} = -\mathbf{x}$. If $A_i(G)\mathbf{x} = \lambda \mathbf{x}$ with $\lambda \neq 0$ for some $i \in \{1, \dots, \operatorname{diam}(G)\}$, then $G \boxtimes K_n$ is CDDR but not DP.
\end{Corollary}

\begin{Example}
    The cycle $C_6$ admits an eigenvector $\mathbf{x}$ such that $A_1(C_6)\mathbf{x} = A_2(C_6)\mathbf{x} = -\mathbf{x}$. Since $C_6$ and $K_n$ are DP, their product $C_6 \boxtimes K_n$ is CDDR and DMR, but not DP for all $n \geq 2$. This demonstrates that the classes of DR and DP graphs are not closed under the strong product.
\end{Example}


\section{Cartesian Product}

We extend the structural characterizations of distance-related regularities under the Cartesian product $G \square H$, building upon the foundations laid by \cite{conde2026}. Previous work established that the CDDR property is preserved under the Cartesian product and explored specific structural obstructions using the link graph $K_2$ \cite{conde2026}.

However, a general converse for arbitrary factors has remained elusive. Our objective is to transcend these localized instances by introducing a distance matrix convolution framework. This approach proves that the inheritance of CDDR is an algebraic invariant of both the product and its factors. Furthermore, we derive new sharpness bounds for spectral determination (DP), completing the global picture of the structural properties established in \cite{conde2026}.

\begin{Proposition}[\cite{conde2026}, Proposition~6]\label{Prop: cddr square cddr = cddr}
Let $G$ and $H$ be two CDDR graphs. Then $G \square H$ is a CDDR graph.
\end{Proposition}

\begin{Proposition}[\cite{conde2026}, Proposition~7]
Let $G$ be a CDDR graph. Then $G \square K_2$ is a CDDR graph. Moreover, if $G \square K_2$ is a DR graph, then $G$ is a DR graph.
\end{Proposition}

\begin{Proposition}[\cite{conde2026}, Proposition~12]
Let $G$ be a DDR graph that is not a CDDR graph. Then $G \square K_2$ is also a DDR graph that is not a CDDR graph.
\end{Proposition}

\begin{Proposition}[\cite{conde2026}, Proposition~18]
Let $G$ be a DDR graph that is not a DMR graph. Then $G \square K_2$ is a DDR graph and is not a DMR graph.
\end{Proposition}

We extend the structural characterizations of distance-related regularities under the Cartesian product $G \square H$. The distance metric $d_{G \square H}((u,x),(v,y)) = d_G(u,v) + d_H(x,y)$ implies that the $i$-th distance matrix satisfies the following formula.

\begin{lemma}
\label{lem:formula_cartesiana}
Let $G$ and $H$ be connected graphs. The $i$-th distance matrix of the Cartesian product $G \square H$ satisfies
\begin{equation}
    A_i(G \square H) = \sum_{k=0}^{i} A_k(G) \otimes A_{i-k}(H),
\end{equation}
where $A_0 = I$, and $A_k(G)$, $A_{i-k}(H)$ denote the respective distance matrices of the factor graphs.
\end{lemma}

This is a matrix reformulation of Theorem 2.2 in \cite{Bresar2019}, which describes the structure of exact distance-p graphs of the Cartesian product in the language of graph unions and direct products. We provide the Kronecker product formulation, which is the form required for the algebraic arguments of this section.

By Lemma~\ref{lem:formula_cartesiana}, verifying that the families of DDR, CDDR and DMR graphs are closed  under the Cartesian product is straightforward.

An immediate and powerful consequence of the matrix convolution established in Lemma~\ref{lem:formula_cartesiana} is the simultaneous diagonalization of the global distance layers for CDDR factors. This allows us to explicitly derive the entire eigenvalue spectrum for the Cartesian distance relations.

\begin{Corollary}\label{coro: espectro cartesiano}
Let $G$ and $H$ be CDDR graphs, whose distance matrices $A_k(G)$ and $A_r(H)$ are simultaneously diagonalized by the eigenvector bases $S_G$ and $S_H$, respectively. Then, $S_G \otimes S_H$ simultaneously diagonalizes the $i$-th  distance matrix $A_i(G \square H)$ for all $0 \leq i \leq \operatorname{diam}(G)+ \operatorname{diam}(H)$. 

Furthermore, if $\lambda_j^{(k)}$ is an eigenvalue of $A_k(G)$ associated with an eigenvector $\mathbf{x}$, and $\mu_m^{(i-k)}$ is an eigenvalue of $A_{i-k}(H)$ associated with an eigenvector $\mathbf{y}$. Then, the eigenvalue $\theta_{j,m}^{(i)}$ of the distance matrix $A_i(G \square H)$ associated with the eigenvector $\mathbf{x} \otimes \mathbf{y}$ is given by
\begin{equation*}
     \theta_{j,m}^{(i)}=\sum_{k=0}^{i} \lambda_j^{(k)} \mu_m^{(i-k)}.
\end{equation*}
\end{Corollary}

\begin{theorem}\label{teo: reciprocos cartesiano}
    If $G \square H$ belongs to one of the classes \emph{DDR, CDDR, DMR, DP or DR}, then both $G$ and $H$ must also belong to that same class.
\end{theorem}
\begin{proof}
    For \textbf{DDR}, the shell size follows the expansion $$|\Gamma_i(u,x)| = \sum_{k=0}^i |\Gamma_k(u)||\Gamma_{i-k}(x)|$$. By strong induction, terms for $k < i$ cancel out, isolating $|\Gamma_i(u)| = |\Gamma_i(v)|$.

    For \textbf{CDDR}, the intersection of distance layers satisfies the convolution-like relation:
    \begin{equation*}
        |\Gamma_i(u,x) \cap \Gamma_j(v,x)| = \sum_{k=0}^i |\Gamma_k(x)| \cdot |\Gamma_{i-k}(u) \cap \Gamma_{j-k}(v)|.
    \end{equation*}
    Applying strong induction on $i$ and $j$, the sum is dominated by the boundary terms where $k=0$ (since $\Gamma_0(x)=\{x\}$). Canceling the identical inductive sums for $k > 0$ isolates the CDDR property $|\Gamma_i(u) \cap \Gamma_j(v)| = |\Gamma_j(u) \cap \Gamma_i(v)|$.

    For \textbf{DMR}, the parameter $\omega_{ii}^{G \square H}$ is expressed as:
    \begin{equation*}
        \omega_{ii}^{G \square H}(u,x) = \sum_{k=0}^i |\Gamma_k(u)| \omega^H_{i-k,i-k}(x) + \sum_{k=0}^i |\Gamma_k(x)| \omega^G_{i-k,i-k}(u).
    \end{equation*}
    Once the parameters for $j < i$ are proven vertex-independent by induction, this formula directly implies that $\omega_{ii}^G$ must also be vertex-independent.

    For \textbf{DP}, the eigenvalues of $A_1(G \square H)$ are $\theta = \lambda + \mu$. The proof of the converse follows by total analogy to the strong product case: any spectral collision in one factor graph forces a corresponding collision in the product graph. 

    Finally, for \textbf{DR}, we utilize the intersection convolution:
    \begin{equation*}
        |\Gamma_i(u, x) \cap \Gamma_j(v, y)| = \sum_{k=0}^{\min(i,j)} |\Gamma_k(u) \cap \Gamma_k(v)| \cdot |\Gamma_{i-k}(x) \cap \Gamma_{j-k}(y)|.
    \end{equation*}
    By strong induction on $i$ and $j$, and given the proven DDR and CDDR status of the factors, the independence of the intersection numbers on the vertex choice is isolated, confirming that $G$ and $H$ are DR.
\end{proof}
\begin{Example}
    Consider the Cartesian product $C_4 \square K_4$. Both the cycle $C_4$ and the complete graph $K_4$ are distance-regular, and hence DP graphs. 

    Let $v_2 = (1,1,1,1)^T$ and $v_1 = (1,-1,1,-1)^T$ be eigenvectors of $C_4$. Their respective eigenvalues for the adjacency matrix $A_1(C_4)$ are $2$ and $-2$, while for the distance-2 matrix $A_2(C_4)$, both yield an eigenvalue of $1$. Similarly, let $u_2 = (1,1,1,1)^T$ and $u_1 = (1,-1,0,0)^T$ be eigenvectors of $K_4$. Their respective eigenvalues for $A_1(K_4)$ are $3$ and $-1$. Since the diameter of $K_4$ is $1$, $A_2(K_4)$ is the zero matrix, yielding an eigenvalue of $0$ for both eigenvectors.

    In the Cartesian product $C_4 \square K_4$, the eigenvectors $v_2 \otimes u_1$ and $v_1 \otimes u_2$ share the same eigenvalue for the adjacency matrix $A_1(C_4 \square K_4)$:
    \begin{align*}
        \theta_{1}(v_2 \otimes u_1) &= 2 + (-1) = 1, \\
        \theta_{1}(v_1 \otimes u_2) &= -2 + 3 = 1.
    \end{align*}
    However, applying the structural convolution formula for the distance-2 matrix in the Cartesian product, their eigenvalues diverge:
    \begin{align*}
        \theta_{2}(v_2 \otimes u_1) &= 1 + 0 + (2)(-1) = -1, \\
        \theta_{2}(v_1 \otimes u_2) &= 1 + 0 + (-2)(3) = -5.
    \end{align*}
    Therefore, $C_4 \square K_4$ is not a DP graph.
\end{Example}

As demonstrated by this counterexample, the classes of DP and DR graphs are not closed under the Cartesian product. However, the following theorem establishes a sufficient condition for the Cartesian product of DP graphs to preserve this property.

\begin{theorem}\label{teo: condicion suficiente DP cartesiano}
Let $G$ and $H$ DP graphs. If the map $\Phi: \text{Spec}(G) \times \text{Spec}(H) \to \mathbb{R}$ given by
\begin{equation*}
    \Phi(\lambda, \mu) =  \lambda +\mu
\end{equation*}
is injective, then  $G \square H$ is DP.
\end{theorem}
\begin{proof}
    The result follows by a completely analogous algebraic argument to the one presented in Theorem \ref{teo: condicion suficiente para DP}.
\end{proof}

\begin{Corollary}\label{coro: Kn DP}
    Let $G$ be a DP graph with eigenvalues $\lambda_1 \geq \lambda_2 \geq \dots \geq \lambda_n$. Then, for any $m > \lambda_1 - \lambda_n$, the Cartesian product $G \square K_m$ is a DP graph.
\end{Corollary}
\begin{proof}
    Using that $\operatorname{Spec}(K_m) = \{m-1, -1\}$, we have that $m > \lambda_1 - \lambda_n \geq \lambda_i - \lambda_j$ for all $1 \leq i,j \leq n$. Then,
    \begin{equation*}
        (m-1) + \lambda_j > \lambda_i - 1 \quad \text{for all } 1 \leq i,j \leq n.
    \end{equation*}
    Therefore, by Theorem \ref{teo: condicion suficiente DP cartesiano}, $G \square K_m$ is DP.
\end{proof}

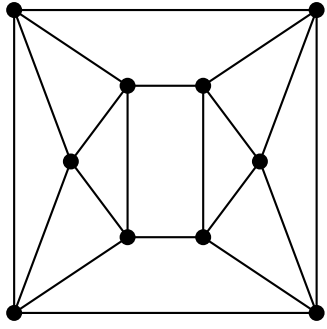
\begin{figure}

\begin{center}
\begin{tikzpicture}[scale=1]
    
    \coordinate (v1) at (-2, 2);
    \coordinate (v2) at (-0.5, 1);
    \coordinate (v3) at (-0.5, -1);
    \coordinate (v4) at (-2, -2);
    \coordinate (v5) at (-1.25, 0);

    \coordinate (w1) at (2, 2);
    \coordinate (w2) at (0.5, 1);
    \coordinate (w3) at (0.5, -1);
    \coordinate (w4) at (2, -2);
    \coordinate (w5) at (1.25, 0);
    
    \foreach \punto in {1,...,5}
        \fill (v\punto) circle(3pt);
	
	\foreach \punto in {1,...,5}
        \fill (w\punto) circle(3pt);
	
%

    \draw[thick] (v1) -- (v2);
    \draw[thick] (v2) -- (v3);
    \draw[thick] (v3) -- (v4);
    \draw[thick] (v4) -- (v1);
    \draw[thick] (v5) -- (v1);
    \draw[thick] (v5) -- (v2);
    \draw[thick] (v5) -- (v3);
    \draw[thick] (v5) -- (v4);
    
    \draw[thick] (w1) -- (w2);
    \draw[thick] (w2) -- (w3);
    \draw[thick] (w3) -- (w4);
    \draw[thick] (w4) -- (w1);
    \draw[thick] (w5) -- (w1);
    \draw[thick] (w5) -- (w2);
    \draw[thick] (w5) -- (w3);
    \draw[thick] (w5) -- (w4);
    
    \draw[thick] (w1) -- (v1);
    \draw[thick] (w2) -- (v2);
    \draw[thick] (w3) -- (v3);
    \draw[thick] (w4) -- (v4);

\end{tikzpicture}
\end{center}

\caption{A DP graph which is not DMR}\label{fig: CDDR y no DMR} 
\end{figure}

In \cite{conde2026}, the authors introduced a graph that is DP, and therefore CDDR, but neither DMR nor DR (see Figure \ref{fig: CDDR y no DMR}). Let us call this graph $C$. We can easily obtain that the spectrum of $C$ is given by $$\operatorname{Spec}(C) = \left\{ 4^1, \left(\frac{1+\sqrt{17}}{2}\right)^1, 1^2, -1^4, \left(\frac{1-\sqrt{17}}{2}\right)^1, -3^1 \right\}.$$

Then, by Corollary \ref{coro: Kn DP} and Theorem \ref{teo: reciprocos cartesiano}, the Cartesian product $C \square K_m$ is DP and not DMR for $m > 7$. The next statement makes this observation precise.

\begin{Corollary}
Let $n \in \mathbb{N}$. There exists a CDDR graph $G$ with $|V(G)| \geq n$ such that $G$ is DP but not DMR.
\end{Corollary}


\section{Direct Product}

Having analyzed the metric structures of the strong and Cartesian products, we now turn our attention to the direct product (also known as the tensor or categorical product), denoted by $G \times H$. While the previous operations exhibit robust preservation of regularity properties, the direct product is inherently more restrictive due to the parity constraints imposed on its walk structure. 

The following theorem provides a comprehensive counterexample to the preservation of several key regularity classes. To establish this, we rely on two distinct structural examples: the graph $C$ (illustrated in Figure 4), which is known to be DP, and the DMR graph presented by \cite{DiFi2017} (illustrated in Figure 5).

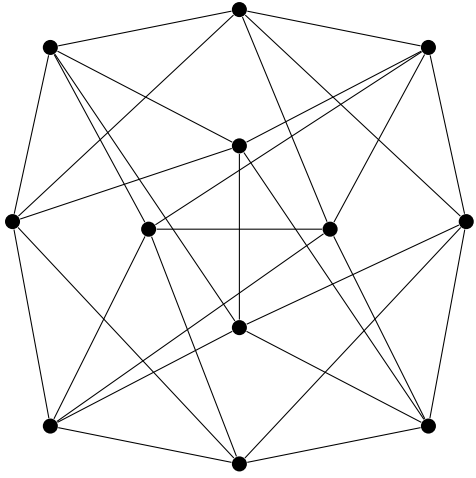
\begin{figure}[h]
\begin{center}
    \begin{tikzpicture}[scale=1, every node/.style={circle, fill=black, inner sep=2pt}]
    \node (n10) at (0, 3) {};
    \node (n3)  at (-2.5, 2.5) {};
    \node (n11) at (2.5, 2.5)  {};
    \node (n1)  at (0, 1.2) {};
    \node (n2)  at (-3, 0.2)  {};
    \node (n9)  at (-1.2, 0.1) {};
    \node (n12) at (1.2, 0.1)  {};
    \node (n6)  at (3, 0.2)  {};
    \node (n4)  at (0, -1.2)  {};
    \node (n7)  at (-2.5, -2.5)  {};
    \node (n5)  at (2.5, -2.5)  {};
    \node (n8)  at (0, -3)  {};

    \draw (n10) -- (n3); \draw (n10) -- (n11); 
    \draw (n10) -- (n12); \draw (n10) -- (n6); \draw (n10) -- (n2);
    \draw (n3) -- (n2); \draw (n3) -- (n1); \draw (n3) -- (n9); \draw (n3) -- (n4);
    \draw (n11) -- (n6); \draw (n11) -- (n1); \draw (n11) -- (n12); \draw (n11) -- (n9);
    \draw (n2) -- (n7); \draw (n2) -- (n8); \draw (n2) -- (n1);
    \draw (n6) -- (n5);  \draw (n6) -- (n8); \draw (n6) -- (n4);
    \draw (n7) -- (n8); \draw (n7) -- (n9); \draw (n7) -- (n12); \draw (n7) -- (n4);
    \draw (n5) -- (n8); \draw (n5) -- (n12); \draw (n5) -- (n1); \draw (n5) -- (n4);
    \draw (n1) -- (n4);
    \draw (n9) -- (n12); \draw (n9) -- (n8);
\end{tikzpicture}
\end{center}
\caption{A DMR which is not vertex-transitive.}\label{fig: contraejemplo producto directo}
\end{figure}

\begin{theorem}
    The classes of DDR, CDDR, DMR, and DP graphs are not closed under the direct product operation $G \times H$.
\end{theorem}

\begin{proof}
    To prove the non-closure of these classes, we examine the direct product of the complete graph $K_2$ with our two specified counterexamples. Note that $K_2$ is distance-regular, representing the highest tier of metric regularity.

    First, consider the graph $C$ defined in Figure 4. Because $C$ is a DP graph, it is inherently CDDR and DDR. However, a direct inspection of the distance-2 matrix $A_2(C \times K_2)$ reveals that the all-ones vector $\mathbf{1}$ is no longer an eigenvector. This algebraic obstruction confirms that the product $C \times K_2$ is not DDR. Since DDR is a necessary condition for both the CDDR and DP properties, $C \times K_2$ fails to inherit these regularities, proving that the DDR, CDDR, and DP classes are not closed under the direct product.

    Second, to address the DMR class, we utilize the graph presented in \cite{DiFi2017} (see Figure 5), which constitutes an example of a DMR graph that lacks vertex-transitivity. It is straightforward to verify that the direct product of this graph with $K_2$ also fails to preserve Distance-Degree Regularity (DDR), as the all-ones vector $\mathbf{1}$ is again not an eigenvector of its distance matrices. Consequently, this product cannot be DMR, confirming that the class of DMR graphs is also not closed under the direct product operation.
\end{proof}

\subsection{Distance Walk Regular Graphs}

The failure of the products $C \times K_2$ and the aforementioned DMR graph to preserve even the most basic degree regularity suggests that the standard definitions of CDDR or DDR are insufficient to guarantee inheritance under the direct product operation. The sensitivity of the direct product to path parity requires a more robust form of structural invariance. This motivates the following definition.

\begin{Definition}
    A connected graph $G$ is said to be \textbf{Distance Walk Regular (DWR)} if, for any two vertices $u, v \in V(G)$, the existence of a walk of length $k$ between $u$ and $v$ depends only on the distance $d(u,v)$. Equivalently, $G$ is DWR if for every $k \in \mathbb{N}$, the support of the $k$-th power of the adjacency matrix belongs to the span of the distance matrices:
    \begin{equation*}
        \operatorname{supp}(A_1^k) \in \operatorname{span}\mathcal{A},
    \end{equation*}
    where $\mathcal{A}=\{I, A_1, \dots, A_d\}$ is the set of distance matrices of $G$.
\end{Definition}

Before exploring the preservation of regularity under the product $G \times H$, we establish the position of DWR graphs within the existing hierarchy of metric regularity.

\begin{theorem}
    Every DR graph is DWR.
\end{theorem}

\begin{proof}
    Since $G$ is DR, its distance matrices $\mathcal{A} = \{I, A_1(G), \dots, A_d(G)\}$ form the basis of the Bose-Mesner algebra. Because this algebra is closed under standard matrix multiplication, any power of the adjacency matrix inherently belongs to $\operatorname{span}(\mathcal{A})$. Thus, for any $k \in \mathbb{N}$, we can uniquely express $A_1(G)^k$ as:
    \begin{equation*}
        A_1(G)^k = \sum_{i=0}^d b_i A_i(G).
    \end{equation*}
    The matrix $A_1(G)^k$ counts the number of walks of length $k$ between any pair of vertices, ensuring that all entries are non-negative. Consequently, $b_i \geq 0$ for all $i$. Furthermore, the distance matrices $A_i(G)$ are $\{0,1\}$-matrices with pairwise disjoint supports. Therefore, applying the support function to $A_1(G)^k$ simply projects the coefficients, mapping any strictly positive $b_i > 0$ to $1$, while null coefficients remain $0$. 
    
    This yields $\operatorname{supp}(A_1(G)^k) = \sum_{i=0}^d \delta_i A_i(G)$ for some $\delta_i \in \{0,1\}$, which explicitly shows that $\operatorname{supp}(A_1(G)^k) \in \operatorname{span}(\mathcal{A})$. Thus, $G$ is DWR.
\end{proof}

\begin{theorem}
    Every bipartite graph is a DWR graph.
\end{theorem}

\begin{proof}
    Let $G$ be a bipartite graph and let $u, v \in V(G)$ with $d(u,v) = i$. Any walk between $u$ and $v$ must have the same parity as their distance $i$, as otherwise the concatenation of such a walk with a shortest path would form an odd closed walk, contradicting that $G$ is bipartite. Thus, no walk of length $j$ exists if $j \not\equiv i \pmod 2$. 
    
    Conversely, for any non-negative integer $k$, a walk of length $i+2k$ trivially exists by traversing a shortest path from $u$ to $v$ and oscillating $k$ times along its final edge. Consequently, the existence of a walk of length $m$ between any two vertices depends strictly on whether $m \geq d(u,v)$ and $m \equiv d(u,v) \pmod 2$. Therefore, $G$ is DWR.
\end{proof}

The strength of the DWR condition lies in its ability to capture structural symmetry without requiring strict degree regularity. For instance, path graphs $P_n$ for $n \geq 3$ are bipartite (and hence DWR) but trivially fail to be regular. This guarantees the existence of arbitrary large distance walk regular structures that lack metric regularity:

\begin{Corollary}
    For any $n, d \in \mathbb{N}$, there exists a graph $G$ with $|V(G)| \geq n$ and $\operatorname{diam}(G) \geq d$ such that $G$ is DWR but not DDR.
\end{Corollary}

\begin{theorem}\label{teo: condicion dwr}
    Let $G$ and $H$ be graphs, not necessarily connected, without isolated vertices. Then their strong product $G \boxtimes H$ is DWR.
\end{theorem}

\begin{proof}
    Let $(u,x), (v,y) \in V(G \boxtimes H)$ such that $d_{G \boxtimes H}((u,x), (v,y)) = i$. We must show that for any $j \geq i$, there exists a walk of length $j$ between these vertices. As in the bipartite case, the existence of a walk of length $i+2k$ for $k \geq 0$ is trivially guaranteed by traversing a shortest path $P$ and oscillating along its final edge. To cover walks of odd length relative to $i$ (i.e., $j = i+2k+1$), it suffices to construct a walk of length $i+1$.

    Let $P = ((u,x) = w_0, w_1, \dots, w_{i-1}, w_i = (v,y))$ be a shortest path of length $i$. The penultimate vertex $w_{i-1} = (w,z)$ is adjacent to $(v,y)$ in the strong product. By definition, they differ in at least one coordinate. We construct a length-2 sub-walk between them as follows:
    If $w \sim v$ and $z \sim y$, we can simply insert the mixed coordinate vertex to form the sub-walk $(w,z), (w,y), (v,y)$. If they differ in only one coordinate, say $w=v$ and $z \sim y$, the absence of isolated vertices in $G$ guarantees a neighbor $v_0 \sim v$, allowing the sub-walk $(w,z), (v_0, y), (v,y)$. The remaining case ($w \sim v$ and $z=y$) is symmetric. 
    
    Replacing the final edge of $P$ with the corresponding length-2 sub-walk yields a valid walk of length $i+1$. By oscillating the final edge of this new walk, we guarantee the existence of walks of length $i+2k+1$. Thus, the existence of a walk depends solely on $j \geq i$, proving that $G \boxtimes H$ is DWR.
\end{proof}

\begin{theorem}\label{teo: condicion no DWR}
    Let $G$ be a graph with odd girth $g(G) = 2s+1$. Then the Cartesian product $G \square K_2$ is not DWR.
\end{theorem}

\begin{proof}
    Let $V(K_2)=\{0,1\}$. Select two adjacent vertices $u,v \in V(G)$ that belong to a minimum odd cycle of length $2s+1$ in $G$. In the Cartesian product, $(u,0) \sim (v,0)$ and $(u,0) \sim (u,1)$, meaning $d((u,0),(v,0)) = 1 = d((u,0),(u,1))$. 
    
    Following the remaining vertices of the minimum cycle entirely within the copy $G \square \{0\}$, there exists a walk of length $2s$ between $(u,0)$ and $(v,0)$. Suppose, for the sake of contradiction, that $G \square K_2$ is DWR. Since the distances are identical, there must also exist a walk $W$ of length $2s$ between $(u,0)$ and $(u,1)$. 
    
    Adding the edge $(u,1)(u,0)$ to $W$, we obtain a closed walk $S$ of odd length $2s+1$ that starts and ends at $(u,0)$. Because $S$ visits $(u,1)$ and must return to $(u,0)$, it must traverse the cross edges between the copies $G \square \{0\}$ and $G \square \{1\}$ a strictly positive, even number of times (say, $2k$ times where $2k \geq 2$). 
    
    Let $\pi_G : G \square K_2 \to G$ be the natural projection map defined by $\pi_G(x,i) = x$. The projection $\pi_G(S)$ drops the $2k$ cross edges, yielding an odd closed walk entirely within $G$ of length exactly $(2s+1) - 2k$. Since $2k \geq 2$, this walk has an odd length of at most $2s-1$. It is a standard result that any odd closed walk must contain an odd cycle, which implies $G$ contains an odd cycle of length $m \leq 2s-1$. However, this directly contradicts the hypothesis that the odd girth of $G$ is $2s+1$. Therefore, $G \square K_2$ is not DWR.
\end{proof}

To illustrate the independence of these regularity classes, we provide a structured catalog of infinite families that selectively satisfy these properties:

First, we consider graphs that are \textbf{DP, DMR, and DWR, but not DR}. By Theorems \ref{teo: condicion suficiente para DP} and \ref{teo: condicion dwr}, the strong product $C_4 \boxtimes K_n$ guarantees the DP and DWR properties for every $n \geq 2$. However, evaluating the vertex intersections for $V(C_4)=\{c_1,c_2,c_3,c_4\}$ and $V(K_n)=\{0,1,\dots,n-1\}$ yields $|\Gamma_1(c_1,0) \cap \Gamma_1(c_1,1)| = 3n-2$ and $|\Gamma_1(c_1,0) \cap \Gamma_1(c_2,0)| = 2n-2$. Since these intersection numbers differ for pairs at distance 1, the graph explicitly fails to be DR.

Second, we isolate graphs that are \textbf{DP and DMR, but not DWR}. The Cartesian product $K_n \square K_2$ (for $n \geq 3$) retains the DP and DMR properties from its factors. However, since $K_n$ contains an odd cycle of length 3 (for $n \geq 3$), Theorem \ref{teo: condicion no DWR} implies that this Cartesian product is strictly not DWR.

Third, we construct graphs that are \textbf{DP, but neither DMR nor DWR}. The graph $C$ (Figure \ref{fig: CDDR y no DMR}) inherently fails to be DWR due to a structural asymmetry: there exist adjacent pairs with a common neighbor (forming a length-2 walk) and adjacent pairs without one. This asymmetry projects cleanly into the Cartesian product $C \square K_n$. As previously established, for $n \geq 8$, $C \square K_n$ is a DP graph but fails to be DMR, while simultaneously remaining non-DWR.

Fourth, we generate graphs that are \textbf{DP and DWR, but not DMR}. Taking the non-DMR graph from the previous case, we observe that $-1 \notin \operatorname{Spec}(C \square K_8)$. Applying Corollary \ref{coro: Kn DP} and Theorem \ref{teo: condicion dwr}, the strong product $(C \square K_8) \boxtimes K_n$ preserves the DP property and strictly forces the DWR condition for every $n \geq 2$. However, it fails to inherit the DMR property.

\begin{Corollary}
    The classes of DWR, DP, CDDR, DMR, and DDR graphs are distinct and exhibit non-empty intersections. Specifically, there exist infinite families of graphs in each possible intersection and symmetric difference of these classes.
\end{Corollary}

\begin{proof}
    From the previously established results, we have already shown that the class of DP graphs contains infinite families of graphs that are simultaneously DWR and DMR, DWR but not DMR, DMR but not DWR, and neither DWR nor DMR.
    
    To demonstrate the existence of infinite families for any other arbitrary intersection or symmetric difference of these classes, let $G$ be a graph belonging to the desired region. By Theorem \ref{teo: condicion dwr} and our strong product preservation theorems, the graph $G \boxtimes K_n$ inherits the exact same class memberships for any $n \geq 2$, with the added guarantee that it is a DWR graph. Since $n \geq 2$ is arbitrary, this construction immediately generates an infinite family of such graphs.
    
    Conversely, note that the strong product $G \boxtimes K_n$ has an odd girth of $3$. By Theorem \ref{teo: condicion no DWR} and our Cartesian product preservation theorems, the graph $(G \boxtimes K_n) \square K_2$ preserves the same class memberships regarding the DP, CDDR, DMR, and DDR properties, but it is explicitly not a DWR graph. Varying $n \geq 2$ again provides an infinite family for these strictly non-DWR conditions. This completes the proof.
\end{proof}

\subsection{Direct Product and DWR Graphs}

To formalize the distance structure of $G$, we introduce the following recursive characterization of its distance matrices. Let $\operatorname{supp}(M)$ denote the support of a matrix $M$.

\begin{lemma}\label{lema: expresion matriz de distancia}
Let $G$ be a graph. For each $i \in \{1, \dots, \operatorname{diam}(G)\}$, the $i$-th distance matrix $A_i(G)$ can be expressed as:
\begin{equation*}
    A_i(G) = \operatorname{supp}(A_1(G)^i) \circ \left(J - \sum_{k=0}^{i-1} A_k(G)\right).
\end{equation*} 
\end{lemma} 

\begin{proof}
    Let $u,v \in V(G)$. We evaluate the $uv$-entry of the right-hand side based on the actual distance $d(u,v)$. 
    
    If $d(u,v) = i$, there is at least one walk of length $i$ between $u$ and $v$, meaning $\operatorname{supp}(A_1(G)^i)_{uv} = 1$. Moreover, since $d(u,v) \not< i$, the pair has not been indexed in any lower-distance matrix, so $\sum_{k=0}^{i-1} (A_k(G))_{uv} = 0$. The Hadamard product yields $1 \cdot (1 - 0) = 1$.
    
    If $d(u,v) < i$, forces $\sum_{k=0}^{i-1} (A_k(G))_{uv} = 1$. The second factor becomes $(1 - 1) = 0$, vanishing the product.
    
    If $d(u,v) > i$, no walk of length $i$ exists between the vertices, forcing the first factor $\operatorname{supp}(A_1(G)^i)_{uv} = 0$, which again vanishes the product. 
    
    In all cases, the equation correctly evaluates to exactly $1$ when $d(u,v)=i$ and $0$ otherwise, matching the definition of $A_i(G)$.
\end{proof}

\begin{Remark}\label{obs: forma de Ai}
    It is straightforward to verify that $\operatorname{supp}(A \otimes B) = \operatorname{supp}(A) \otimes \operatorname{supp}(B)$ for any non-negative real matrices $A$ and $B$. Since $A_1(G \times H) = A_1(G) \otimes A_1(H)$, applying standard Kronecker product properties to Lemma \ref{lema: expresion matriz de distancia} yields:
    \begin{equation*}
        A_i(G \times H) = \left( \operatorname{supp}(A_1(G)^i) \otimes \operatorname{supp}(A_1(H)^i) \right) \circ \left( J_G \otimes J_H - \sum_{k=0}^{i-1} A_k(G \times H) \right).
    \end{equation*}
\end{Remark}

\begin{lemma}\label{lema: Ai span tensor}
    Let $G$ and $H$ be DWR graphs such that at least one is non-bipartite. Then, $A_i(G \times H) \in \operatorname{span}(\mathcal{A}_G \otimes \mathcal{A}_H)$ for every $i \in \{1, \dots, \operatorname{diam}(G \times H)\}$.
\end{lemma}

\begin{proof}
    We proceed by strong induction on $i$. The base cases $i \in \{0, 1\}$ trivially belong to the span. Assume the result holds for all $j < i$. By Remark \ref{obs: forma de Ai}, $A_i(G \times H)$ is obtained via matrix addition and the Hadamard product of $\operatorname{supp}(A_1(G)^i) \otimes \operatorname{supp}(A_1(H)^i)$ with lower-index distance matrices. Since $G$ and $H$ are DWR, their respective support matrices lie in $\operatorname{span}(\mathcal{A}_G)$ and $\operatorname{span}(\mathcal{A}_H)$. Invoking the inductive hypothesis, the mixed-product property for Hadamard and Kronecker products, and Lemma \ref{lema: AoB esta en el span}, every term in the expansion is algebraically closed under the tensor span. Thus, $A_i(G \times H) \in \operatorname{span}(\mathcal{A}_G \otimes \mathcal{A}_H)$.
\end{proof}

\begin{theorem}
    Let $G$ and $H$ be graphs belonging to one of the regularity classes \emph{DDR, CDDR, or DMR}. If both $G$ and $H$ are DWR and at least one is non-bipartite, then the direct product $G \times H$ belongs to the same respective class.
\end{theorem}

\begin{proof}
    The result follows immediately from Lemma \ref{lema: Ai span tensor}. Since every distance matrix $A_i(G \times H)$ resides strictly within $\operatorname{span}(\mathcal{A}_G \otimes \mathcal{A}_H)$, the structural constants (such as row sums, entrywise products, and diagonal traces) of the product graph are entirely determined by the corresponding vertex-independent parameters of the factor matrices. Consequently, $G \times H$ directly inherits the regularity properties of its factors.
\end{proof}

\begin{Corollary}
    Let $G$ and $H$ be CDDR and DWR graphs, where at least one is non-bipartite. Let $S_G$ and $S_H$ be the matrices that simultaneously diagonalize the distance matrices of $G$ and $H$, respectively. Then, $S_G \otimes S_H$ simultaneously diagonalizes $A_k(G \times H)$ for all $1 \leq k \leq \operatorname{diam}(G \times H)$.
\end{Corollary}

\begin{proof}
    By the preceding theorem, $G \times H$ is CDDR, ensuring its distance matrices commute and are simultaneously diagonalizable. Furthermore, by Lemma \ref{lema: Ai span tensor}, every distance matrix $A_k(G \times H)$ belongs to $\operatorname{span}(\mathcal{A}_G \otimes \mathcal{A}_H)$. Since $S_G$ and $S_H$ diagonalize the basis matrices of $\mathcal{A}_G$ and $\mathcal{A}_H$, standard properties of the Kronecker product dictate that $S_G \otimes S_H$ diagonalizes any matrix within this tensor span, completing the proof.
\end{proof}

    In the study of DWR graphs, we have shown that they constitute a remarkably broad family with significant structural variability. A natural question arises: is it worth defining and isolating a class of graphs where the number of walks of any length $k$ between two vertices $u$ and $v$ depends solely and exclusively on the distance between $u$ and $v$?
    
    In matrix terms, this walk-regularity property is equivalent to requiring that every power of the adjacency matrix, $A_1^k$, belongs to the vector space spanned by the distance matrices of $G$. That is, $A_1^k \in \operatorname{span}\{A_0, A_1, \dots, A_d\}$. This algebraic imposition, far from being restrictive, yields the following fundamental characterization.
    
    \begin{theorem}
    Let $G$ be a graph. Then $G$ is distance regular if and only if $A_1^k \in \operatorname{span}\{A_0, A_1, \dots, A_d\}$ for every integer $k \geq 1$. More precisely, $G$ is DR if and only if this condition holds for $1 \leq k \leq 2d$.
    \end{theorem}
    \begin{proof}
    Suppose that $A_1^k \in \operatorname{span}\{A_0, A_1, \dots, A_d\}$. Then, for $k \leq d$, we can write:$$A_1^k = c_0 I + c_1 A_1 + \dots + c_k A_k + \dots + c_d A_d$$Taking the Hadamard product of both sides with $A_i$ for $i > k$, we obtain:$$A_1^k \circ A_i = c_i A_i$$-
    
    However, the left side is the zero matrix, $\mathbf{0}$, because there is no pair of vertices $u$ and $v$ at distance $i$ connected by a walk of length $k$ (since $k < i$). Consequently, $c_i = 0$ for every $i > k$.
    
    Moreover, $A_1^k \circ A_k = c_k A_k$. Since $k \leq d$, there is at least one pair of vertices $u, v$ such that $d(u,v) = k$, meaning there is a walk of length $k$ between them. Thus, $c_k > 0$.Then, for $k \leq d$, we can isolate $A_k$:$$A_k = \frac{1}{c_k}A_1^k - \left(\frac{c_0}{c_k}I + \frac{c_1}{c_k}A_1 + \frac{c_2}{c_k}A_2 + \cdots + \frac{c_{k-1}}{c_k}A_{k-1}\right)$$
    
    By strong induction, it is straightforward to see that $A_k = P_k(A_1)$, where $P_k$ is a polynomial of degree exactly $k$. This implies that $G$ is a DP graph.Now, let $0 \leq i, j \leq d$. Since $A_i$ and $A_j$ are polynomials in $A_1$, their product is also a polynomial in $A_1$:$$A_i A_j = P_i(A_1) P_j(A_1) = R(A_1)$$where $R$ is a polynomial of degree $i+j \leq 2d$. Thus, $A_i A_j$ can be expressed as a linear combination of powers of $A_1$: $$A_i A_j = a_0 I + a_1 A_1 + \cdots + a_{2d} A_1^{2d}$$
    Since we assumed $A_1^k \in \operatorname{span}\{A_0, A_1, \dots, A_d\}$ for all $1 \leq k \leq 2d$, it inherently follows that $A_i A_j \in \operatorname{span}\{A_0, A_1, \dots, A_d\}$.Therefore, the distance matrices span an algebra (the Bose-Mesner algebra), proving that $G$ is DR. The converse is trivial from the definition of distance-regularity, completing the proof.
    \end{proof}
    
    \begin{Remark}
    The upper bound of $2d$ in the theorem provides a significant computational advantage. Verifying the DR property through the standard definition requires computing and checking the span of $A_i A_j$ for all pairs, which involves $\mathcal{O}(d^2)$ matrix multiplications. In contrast, this characterization reduces the computational cost to evaluating exactly $2d$ powers of the adjacency matrix $A_1$.
    \end{Remark}


\section{Lexicographic Product}

In this section, we study the behavior of distance commutativity under graph operations. Specifically, we focus on the lexicographic product of graphs (also known as graph composition). Our primary objective is to extend and generalize a recent result established by Conde et al \cite{conde2026}, which states the following:
\begin{Proposition}[\cite{conde2026}, Proposition~4]
Let $G$ be a regular graph. Then $C_n[G]$ is a CDDR graph for any $n \geq 3$.
\end{Proposition}
By analyzing the block structure of the distance matrices of the lexicographic product $G[H]$, we aim to relax the condition on the base graph from a cycle $C_n$ to any arbitrary CDDR graph, thereby providing a broader framework for constructing infinite families of distance-commutative graphs. 

In addition to the CDDR property, we also explore how this graph composition behaves with respect to other spectral and degree regularities. In particular, we determine whether the combination of a base graph and a regular factor graph remains stable under the Distance-Degree Regular (DDR), Distance-Matrix Regular (DMR), and Distance-Polynomial (DP) classifications.

It is well known that the adjacency matrix of the lexicographic product of two graphs, $G$ and $H$, can be expressed in terms of the Kronecker product as
$$A_1(G [H] ) = A_1(G) \otimes J_{|V(H)|} + I_{|V(G)|} \otimes A_1(H),$$

Then, a matrix reformulation of Theorem 2.5 in \cite{Bresar2019} gives us the following result.

\begin{lemma} \label{lem:dist_matrices_lexicographic}
Let $G$ and $H$ graphs, where $|V(G)| \geq 2$. Let $n = |V(G)|$ and $m = |V(H)|$. Then the $i$-th distance matrix of the lexicographic product $G[H]$ satisfies:
$$A_1(G[H]) = A_1(G) \otimes J_m + I_n \otimes A_1(H),$$
$$A_2(G[H]) = A_2(G) \otimes J_m + I_n \otimes (J_m - I_m - A_1(H)),$$
and for any distance layer $i \geq 3$,
$$A_i(G[H]) = A_i(G) \otimes J_m.$$
\end{lemma}

\begin{theorem}
    Let $G$ be a graph with $|V(G)| \geq 2$, and let $H$ be a regular graph. If $G$ is a DDR (respectively, CDDR) graph, then the lexicographic product $G[H]$ is also DDR (respectively, CDDR).
\end{theorem}
\begin{proof}
    Suppose $G$ is DDR. Since $H$ is regular, $\mathbf{1}_m$ is a common eigenvector of both $A_1(H)$ and $J_m$. By Lemma \ref{lem:dist_matrices_lexicographic} and the mixed-product property of the Kronecker product, it follows directly that the global all-ones vector $\mathbf{1}_n \otimes \mathbf{1}_m$ is an eigenvector of $A_i(G[H])$ for all $i \geq 1$, proving $G[H]$ is DDR.

    Similarly, assume $G$ is CDDR. The regularity of $H$ ensures that $A_1(H)$ commutes with $J_m$ (and trivially with $I_m$). Since the distance matrices of $G$ mutually commute by hypothesis, all tensor components in Lemma \ref{lem:dist_matrices_lexicographic} mutually commute. Applying the mixed-product property yields that $A_i(G[H]) A_j(G[H]) = A_j(G[H]) A_i(G[H])$ for all valid distance layers $i, j$. Thus, $G[H]$ is CDDR.
\end{proof}

\begin{Corollary} \label{cor:eigenvalues_lexicographic_final}
    Let $G$ be a CDDR graph and $H$ be a regular graph on $m$ vertices with degree $k_H$, such that $S_G$ is the matrix that simultaneously diagonalizes $A_i(G)$ for all $1 \leq i \leq \operatorname{diam}(G)$, and $S_H$ is a matrix that diagonalizes $A_1(H)$. Then, $S_G \otimes S_H$ simultaneously diagonalizes $A_k(G[H])$ for all $1 \leq k \leq \operatorname{diam}(G[H])$.

    Moreover, let $\lambda_j^{(k)}$ be the eigenvalue of $A_k(G)$ associated with the eigenvector $\mathbf{x}_j$, and let $\mu_r$ be the eigenvalue of $A_1(H)$ associated with the eigenvector $\mathbf{y}_r$. Assume $\mathbf{y}_1 = \mathbf{1}_m$ (thus $\mu_1 = k_H$) and the remaining eigenvectors $\mathbf{y}_r$ (for $r \geq 2$) are orthogonal to $\mathbf{1}_m$. Then, the eigenvalue $\theta_{j,r}^{(k)}$ of the distance matrix $A_k(G[H])$ associated with the eigenvector $\mathbf{x}_j \otimes \mathbf{y}_r$ is given by:

    \begin{enumerate}
        \item For the adjacency matrix ($k = 1$):
        \begin{align*}
            \theta_{j,1}^{(1)} &= m \lambda_j^{(1)} + k_H \\
            \theta_{j,r}^{(1)} &= \mu_r \quad \text{for } r \geq 2
        \end{align*}

        \item For the second distance matrix ($k = 2$):
        \begin{align*}
            \theta_{j,1}^{(2)} &= m \lambda_j^{(2)} + m - 1 - k_H \\
            \theta_{j,r}^{(2)} &= -1 - \mu_r \quad \text{for } r \geq 2
        \end{align*}

        \item For long distances ($k \geq 3$):
        \begin{align*}
            \theta_{j,1}^{(k)} &= m \lambda_j^{(k)} \\
            \theta_{j,r}^{(k)} &= 0 \quad \text{for } r \geq 2
        \end{align*}
    \end{enumerate}
\end{Corollary}

\begin{Corollary}

    Let $G$ be a DP graph and $H$ be a regular graph on $m$ vertices with degree $k_H$, such that $m \lambda + k_H \notin \operatorname{Spec}(H)$
    for all $\lambda \in \operatorname{Spec}(G)$, then $G[H]$ is a DP graph.
\end{Corollary}
\begin{proof}
    By Corollary \ref{cor:eigenvalues_lexicographic_final}, the eigenvalues of $A_1(G[H])$ are either of the form $m\lambda + k_H$ with $\lambda \in \operatorname{Spec}(G)$, or $\mu \in \operatorname{Spec}(H)$ (specifically those associated with eigenvectors orthogonal to the all-ones vector). 
    
    Suppose there is an equality of two eigenvalues in the adjacency matrix $A_1(G[H])$. Due to the non-resonance condition $m\lambda + k_H \notin \operatorname{Spec}(H)$, a cross-collision between these two distinct forms is impossible. Therefore, the equality must arise strictly from within the same family: either both eigenvalues come from $m\lambda + k_H = m\lambda' + k_H$ (which implies $\lambda = \lambda'$), or both come from $\mu = \mu'$.
    Then using Proposition \ref{prop: caracterizacion de DP} and that $G[H]$ is CDDR is easy to prove that $G[H]$ is DP.
 
\end{proof}

We now turn our attention to the DMR property. However, unlike the previous cases, preserving this regularity requires imposing stronger algebraic conditions on the factor graph $H$.

\begin{theorem}
    Let $G$ be a DMR graph and let $H$ be a regular graph such that every vertex is contained in the same number of triangles. Then, the lexicographic product $G[H]$ is also a DMR graph.
\end{theorem}
\begin{proof}
    By hypothesis, every vertex of $H$ is contained in the same number of triangles, which is algebraically equivalent to stating that the all-ones vector $\mathbf{1}_{|V(H)|}$ is an eigenvector of the matrix $A_1(H)^2 \circ A_1(H)$. 

    Given this property, the regularity of $H$, and the block structure of the distance matrices for the lexicographic product from Lemma~\ref{lem:dist_matrices_lexicographic}, it is easy to see that since $G$ is a DMR graph, the global all-ones vector $\mathbf{1}_{V(G)} \otimes \mathbf{1}_{V(H)}$ is an eigenvector of $A_i(G[H])A_j(G[H]) \circ A_h(G[H])$ for all $h, i, j$. By Lemma~\ref{lem: caracterización de DMR con matrices}, this confirms that $G[H]$ is a DMR graph.
\end{proof}

\begin{theorem}
    If the lexicographic product $G[H]$ belongs to one of the classes DDR, CDDR or DP, then $G$ must also belong to that same class, and $H$ must be a regular graph.
\end{theorem}

\begin{proof}
    Since any DDR or CDDR graph is inherently regular, the regularity of $G[H]$ immediately forces both $G$ and $H$ to be regular graphs. Let $m = |V(H)|$.

    For DDR, the shell size for distance layers $i \geq 3$ follows the direct expansion:
    $$|\Gamma_i((u, x))| = m |\Gamma_i(u)|.$$
    Because $G[H]$ is DDR, the term $|\Gamma_i((u, x))|$ is independent of the chosen vertex $(u, x)$. Canceling the constant factor $m$ immediately isolates $|\Gamma_i(u)| = |\Gamma_i(v)|$ for any $u, v \in V(G)$. For the local cases $i \in \{1, 2\}$, the vertex independence follows analogously.

    For CDDR, the intersection of distance layers for $i, j \geq 3$ simplifies to the relation:
    $$|\Gamma_i((u, x)) \cap \Gamma_j((v, x))| = m |\Gamma_i(u) \cap \Gamma_j(v)|.$$
    Given that $G[H]$ is CDDR, this intersection size is symmetric with respect to the distance indices $i$ and $j$. Dividing both sides by the constant factor $m$ directly isolates the symmetry property $|\Gamma_i(u) \cap \Gamma_j(v)| = |\Gamma_j(u) \cap \Gamma_i(v)|$, confirming that $G$ is CDDR. The boundary intersections involving $i, j \leq 2$ follow analogously.

    Finally, for the DP class, the proof of the converse follows by total analogy to the strong product case. 
\end{proof}


\section{$i$-CDDR Graphs}

As motivated by the open problems discussed in the introduction, we now explore the structural consequences of relaxing the global commutativity conditions of CDDR graphs. This direction leads us to introduce the family of $i$-Commutative Distance Degree-Regular ($i$-CDDR) graphs.

While a CDDR graph requires the entire set of distance matrices $\mathcal{A}_G = \{A_0, A_1, \dots, A_d\}$ to be pairwise commutative, an $i$-CDDR graph demands only that a specific $i$-th distance matrix $A_i$ commutes with every other matrix in the set. This localization allows for a more granular analysis of the algebraic structure of the graph.

\begin{Definition}
A graph $G$ is \textbf{$i$-CDDR} if for any $u, v \in V(G)$ and any integer $1 \leq j \leq \operatorname{diam}(G)$, the following holds:$$|\Gamma_i(u) \cap \Gamma_j(v)| = |\Gamma_i(v) \cap \Gamma_j(u)|.$$

Equivalently, $A_i(G) A_j(G) = A_j(G) A_i(G)$ for all $0 \leq j \leq \operatorname{diam}(G)$.
\end{Definition}

It is immediate from the definition that a graph is CDDR if and only if it is $i$-CDDR for every $1 \leq i \leq \operatorname{diam}(G)$.

\begin{theorem}
    Let $G$ be an $i$-CDDR graph. If the $i$-th distance graph $G_i$ is connected, then $G$ is DDR.
\end{theorem}

\begin{proof}
    By definition, the distance matrix $A_i$ commutes with every distance matrix $A_j$ for $0 \leq j \leq d$. Consequently, $A_i$ commutes with the all-ones matrix $J = \sum_{k=0}^{d} A_k$. This commutativity ($A_i J = J A_i$) implies that $A_i$ has constant row sums, say $k_i$. Thus, $G_i$ is $k_i$-regular, and the all-ones vector $\mathbf{1}$ is an eigenvector of $A_i$ corresponding to the eigenvalue $k_i$.
    
    Now, for any distance matrix $A_j$, consider the vector $\mathbf{v} = A_j \mathbf{1}$. Utilizing the commutativity between $A_i$ and $A_j$, we obtain:
    $$A_i \mathbf{v} = A_i (A_j \mathbf{1}) = A_j (A_i \mathbf{1}) = A_j (k_i \mathbf{1}) = k_i \mathbf{v}$$
    Thus, $\mathbf{v}$ is an eigenvector of $A_i$ for the eigenvalue $k_i$. Since $G_i$ is connected by hypothesis, the matrix $A_i$ is non-negative and irreducible. By the Perron-Frobenius Theorem, the eigenspace associated with the spectral radius $k_i$ has multiplicity one and is spanned by $\mathbf{1}$.
    
      Therefore, the vector $\mathbf{v}$ must be a scalar multiple of $\mathbf{1}$, meaning $A_j \mathbf{1} = \lambda_j \mathbf{1}$ for some scalar $\lambda_j$. This demonstrates that every distance matrix $A_j$ has constant row sums, proving that $G$ is a DDR graph.
\end{proof}

It is worth noting that in the context of simple and connected graphs, the $1$-CDDR property is a sufficiently strong condition to imply DDR. However, the converse is strictly false.

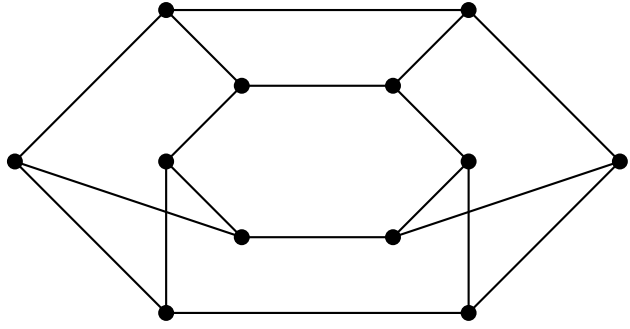
\begin{figure}

\begin{center}
\begin{tikzpicture}[scale=1]
    
    \coordinate (v1) at (-2, 2);
    \coordinate (v2) at (2, 2);
    \coordinate (v3) at (4, 0);
    \coordinate (v4) at (2, -2);
    \coordinate (v5) at (-2, -2);
    \coordinate (v6) at (-4, 0);
    
    \coordinate (w1) at (-1, 1);
    \coordinate (w2) at (1, 1);
    \coordinate (w3) at (2, 0);
    \coordinate (w4) at (1, -1);
    \coordinate (w5) at (-1, -1);
    \coordinate (w6) at (-2, 0);
    
    \foreach \punto in {1,...,6}
        \fill (v\punto) circle(3pt);
	
	\foreach \punto in {1,...,6}
        \fill (w\punto) circle(3pt);
	
%

    \draw[thick] (v1) -- (v2);
    \draw[thick] (v2) -- (v3);
    \draw[thick] (v3) -- (v4);
    \draw[thick] (v4) -- (v5);
    \draw[thick] (v5) -- (v6);
    \draw[thick] (v6) -- (v1);
    
    \draw[thick] (w1) -- (w2);
    \draw[thick] (w2) -- (w3);
    \draw[thick] (w3) -- (w4);
    \draw[thick] (w4) -- (w5);
    \draw[thick] (w5) -- (w6);
    \draw[thick] (w6) -- (w1);
    
    \draw[thick] (w1) -- (v1);
    \draw[thick] (w2) -- (v2);
    
    \draw[thick] (w3) -- (v4);
    \draw[thick] (w4) -- (v3);
    
    \draw[thick] (w5) -- (v6);
    \draw[thick] (w6) -- (v5);
    
\end{tikzpicture}
\end{center}

\caption{A DDR graph which is not $i$-CDDR}\label{fig: DDR y no CDDR} 
\end{figure}

The graph presented in Figure \ref{fig: DDR y no CDDR}, which has diameter $d=3$, serves as a counterexample. An exhaustive calculation verifies that while this graph satisfies the DDR property, its distance matrices fail to commute globally. Specifically, none of the distance matrices commute with the entire set $\{A_0, \dots, A_d\}$ (e.g., $A_i A_j \neq A_j A_i$ for multiple pairs $i \neq j$). Consequently, this graph is DDR but fails to be $i$-CDDR for any $i \in \{1, 2, 3\}$. This explicitly demonstrates the strict inclusion $\{1\text{-CDDR}\} \subsetneq \{\text{DDR}\}$.

\begin{theorem}
    Let $G$ be a graph with $\operatorname{diam}(G) \leq 3$. Then $G$ is $1$-CDDR if and only if $G$ is CDDR.
\end{theorem}

\begin{proof}
    The reverse implication is trivial by definition. Suppose $G$ is a $1$-CDDR graph. To establish that $G$ is CDDR, it suffices to show that $A_2$ and $A_3$ commute. Because $G$ is $1$-CDDR, it is necessarily DDR, implying that every distance matrix commutes with the all-ones matrix $J = I + A_1 + A_2 + A_3$. Since $A_2$ trivially commutes with $J$, $I$, and $A_2$, and commutes with $A_1$ by the $1$-CDDR hypothesis, it must strictly commute with the only remaining term in the decomposition, $A_3$. Therefore, all distance matrices commute pairwise, and $G$ is CDDR.
\end{proof}

Following an exhaustive computational search, we identified a specific Cayley graph that separates the classes of CDDR and $1$-CDDR graphs.

\begin{Example}
    Consider the Cayley graph $G = \text{Cay}(S_5, S)$ generated over the symmetric group $S_5$ (acting on the elements $\{0, 1, 2, 3, 4\}$), with the following generating set:
    \begin{equation*}
    S = \{ (1324), (024)(13), (031)(24), (1243), (1342), (013)(24), (042)(13), (1423) \}
    \end{equation*}
    The resulting graph $G$ has diameter $d=5$. Computational analysis confirms that the adjacency matrix $A_1$ commutes with every distance matrix $A_i$ of the graph. However, the distance algebra is not commutative. Specifically, $A_2$ and $A_3$ do not commute. Furthermore, non-commutation also occurs between the following pairs: $A_2$ and $A_5$, $A_3$ and $A_4$, and $A_4$ and $A_5$. Moreover, $G$ is vertex-transitive so is DMR.

    It is worth noting that $S_5$ is non-abelian and, crucially, non-solvable: its composition series passes through $A_5$, the smallest non-abelian simple group. This structural property appears to be intimately connected with the failure of the implication 1-CDDR $\Rightarrow$ CDDR, suggesting that the solvability of the underlying group may play a decisive role in governing distance commutativity, a point we return to in Section 8.Furthermore, we observe that the generating set $S$ consists entirely of odd permutations. Consequently, the resulting graph is bipartite. Interestingly, this graph is also DWR..
    
    This result demonstrates that the class of CDDR graphs is strictly included in the class of $1$-CDDR graphs and so CDDR graphs is strictly included in the class of $i$-CDDR.
\end{Example}

\begin{theorem}
    Let $G$ and $H$  graphs. The following closure properties hold:
    \begin{enumerate}
        \item $G$ and $H$ are $1$-CDDR if and only if their strong product $G \boxtimes H$ is $1$-CDDR.
        \item $G$ and $H$ are $1$-CDDR if and only if their Cartesian product $G \square H$ is $1$-CDDR.
        \item If $G$ and $H$ are both $1$-CDDR and DWR, and at least one is non-bipartite, then their direct product $G \times H$ is $1$-CDDR.
        \item $G$ is $i$-CDDR and $H$ is regular if and only if $G[H]$ is $i$-CDDR. Where $|V(G)|\geq 2$.
    \end{enumerate}
\end{theorem}

\begin{proof}
    The proofs follow by a completely analogous algebraic argument to those established for the CDDR property. Utilizing the structural formulas for the distance matrices of the respective products, the commutativity of the adjacency matrix $A_1$ distributes over the Kronecker and Hadamard operators. 
    Specifically for the lexicographic product, the regularity of $H$ ensures that its adjacency block commutes with the all-ones matrix, naturally reducing the commutativity of $A_i(G[H])$ directly to the $i$-CDDR property of the base graph $G$.
\end{proof}

The preceding theorem allows for the construction of infinite families of graphs that strictly separate the hierarchies of distance-regularity. By employing the strong product $\boxtimes$ and the Cartesian product $\square$ as lifting operators, we can propagate specific non-commutative behaviors while simultaneously controlling structural parameters such as order and diameter.

\begin{Corollary}
    For any $n, d \in \mathbb{N}$, there exist graphs $G$ with $|V(G)| \geq n$ and $\operatorname{diam}(G) \geq d$ such that:
    \begin{enumerate}
        \item $G$ is DDR but not $1$-CDDR,
        \item $G$ is $1$-CDDR but not CDDR.
    \end{enumerate}
    Furthermore, in both cases, the graph $G$ can be constructed to independently satisfy or fail the DMR and DWR properties.
\end{Corollary}

\begin{theorem}
    Let $G$ be a graph with $\operatorname{diam}(G)=d$ such that $|\Gamma_d(u)|=1$ for every $u \in V(G)$. Let $\sigma$ be the antipodal map. Then $\sigma \in \operatorname{Aut}(G)$ if and only if $G$ is $d$-CDDR. 
\end{theorem}

\begin{proof}
    Since $|\Gamma_d(u)|=1$ for all vertices, the $d$-th distance matrix $A_d$ corresponds exactly to the permutation matrix $P_\sigma$ of the antipodal map $\sigma$. The map $\sigma$ is an automorphism if and only if it preserves all graph distances; expressed algebraically, $P_\sigma A_i P_\sigma^T = A_i$ for all $1 \leq i \leq d$. Because the antipodal map is an involution ($\sigma^2 = \operatorname{id}$), we have $P_\sigma = P_\sigma^T = P_\sigma^{-1}$. Thus, the automorphism condition simplifies directly to $A_d A_i = A_i A_d$. This confirms that $\sigma \in \operatorname{Aut}(G)$ if and only if $A_d$ commutes with the entire distance algebra $\mathcal{A}(G)$, which is precisely the definition of a $d$-CDDR graph.
\end{proof}

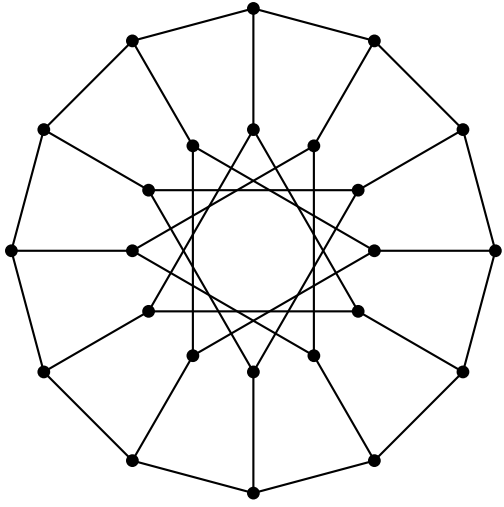
\begin{figure}[h]
\begin{center}
\begin{tikzpicture}[scale=0.8]
    \foreach \i in {1,...,12} {
        \coordinate (v\i) at ({90-(\i-1)*30}:4);
    }
    
    \foreach \i in {1,...,12} {
        \coordinate (w\i) at ({90-(\i-1)*30}:2);
    }
    
    \foreach \punto in {1,...,12}
        \fill (v\punto) circle(3pt);
    
    \foreach \punto in {1,...,12}
        \fill (w\punto) circle(3pt);
        
    \draw[thick] (v1) -- (v2);
    \draw[thick] (v2) -- (v3);
    \draw[thick] (v3) -- (v4);
    \draw[thick] (v4) -- (v5);
    \draw[thick] (v5) -- (v6);
    \draw[thick] (v6) -- (v7);
    \draw[thick] (v7) -- (v8);
    \draw[thick] (v8) -- (v9);
    \draw[thick] (v9) -- (v10);
    \draw[thick] (v10) -- (v11);
    \draw[thick] (v11) -- (v12);
    \draw[thick] (v12) -- (v1);
    
    \foreach \i in {1,...,12} {
        \draw[thick] (v\i) -- (w\i);
    }
    
    \draw[thick] (w1) -- (w5) -- (w9) -- (w1);
    \draw[thick] (w2) -- (w6) -- (w10) -- (w2);
    \draw[thick] (w3) -- (w7) -- (w11) -- (w3);
    \draw[thick] (w4) -- (w8) -- (w12) -- (w4);

\end{tikzpicture}
\end{center}
\caption{The Generalized Petersen Graph $GP(12,4)$, a uniquely antipodal 5-CDDR graph.}\label{fig: GP124}
\end{figure}
The Generalized Petersen graph $GP(12, 4)$ (Figure \ref{fig: GP124}) serves as a definitive counterexample demonstrating that the $d$-CDDR property does not imply Distance-Degree Regularity (DDR). In this graph, the diameter distance matrix $A_5$ strictly commutes with all distance matrices $A_i$, satisfying the $5$-CDDR condition. However, the graph fails to satisfy global metric regularity: the all-ones vector $\mathbf{1}$ is explicitly not an eigenvector of $A_2$. The non-constant row sums of $A_2$ reflect the structural asymmetry between the outer cycle vertices and the internal triangular components. Thus, $GP(12,4)$ confirms that a graph can exhibit localized algebraic centralism at its diameter without being distance-degree regular.


\section{Cayley Graphs}

Cayley graphs offer a rich algebraic framework for studying distance-regularity, providing a lens through which we can explore group theory. It is a well-established result that circulant graphs—specifically, Cayley graphs over cyclic groups $\mathbb{Z}_n$—inherently belong to the class of CDDR graphs. In this section, we extend this structural analysis to broader families of groups, exploring the exact boundaries where distance commutativity begins to break.

To ensure the necessary structural baseline for our metric analysis, we recall the standard construction. Let $H$ be a finite group and $S \subseteq H$ be a generating set such that the identity element $e \notin S$ and $S$ is symmetric ($S = S^{-1}$). This guarantees that the resulting Cayley graph $\operatorname{Cay}(H,S)$ is connected, undirected, and loopless. Because every Cayley graph is vertex-transitive, it inherently satisfies the Distance Mean-Regular (DMR) property, providing an ideal foundation for testing higher-order metric and commutative regularities.

In the study of Cayley graphs, the inherent vertex-transitivity allows us to completely capture the metric structure of the graph by examining only the distance shells around the identity element $e$. Let $\Gamma_i(e) = \{h \in H \mid d(e,h) = i\}$ denote the set of elements at exactly distance $i$ from the identity.

By definition, the chosen generating set $S$ corresponds exactly to the first distance shell, meaning $S = \Gamma_1(e)$. A fundamental consequence of this group-theoretic structure is that for any distance $i \in \{1, \dots, \operatorname{diam}(G)\}$, the $i$-th distance graph $G_i$ is itself a Cayley graph over the same group $H$, generated precisely by the elements of $\Gamma_i(e)$. That is:$$G_i = \operatorname{Cay}(H, \Gamma_i(e))$$
This observation is analytically powerful. It establishes that the $i$-th distance matrix $A_i(G)$ is exactly the adjacency matrix of $\operatorname{Cay}(H, \Gamma_i(e))$. Consequently, verifying the commutative properties of distance matrices in $G$ is entirely equivalent to analyzing the algebraic interactions between the distance-generating sets $\Gamma_i(e)$ within the underlying group. With this structural framework established, we present the following theorem.

\begin{theorem}
Let $G = \operatorname{Cay}(H,S)$ and let $\mathcal{A} = \{A_0, A_1, \dots, A_d\}$ be the distance algebra. Then $A_i A_j = A_j A_i$ if and only if $\Gamma_i(e) \Gamma_j(e) = \Gamma_j(e) \Gamma_i(e)$, where $e$ is the identity element of $H$.
\end{theorem}
\begin{proof}
Suppose that $\Gamma_i(e) \Gamma_j(e) = \Gamma_j(e) \Gamma_i(e)$. We know that $(A_i A_j)_{uv} = |\Gamma_i(u) \cap \Gamma_j(v)|$. For each $x \in \Gamma_i(u) \cap \Gamma_j(v)$, we can write:$$x = r_i u = r_j v \quad \text{where } r_i \in \Gamma_i(e) \text{ and } r_j \in \Gamma_j(e).$$

From this, $v u^{-1} = r_j^{-1} r_i$. Since the graph is undirected, $r_j^{-1} \in \Gamma_j(e)$. Thus, $r_j^{-1} r_i \in \Gamma_j(e) \Gamma_i(e)$. By our hypothesis, we can rewrite this product as $r_j^{-1} r_i = s_i s_j$ for some $s_i \in \Gamma_i(e)$ and $s_j \in \Gamma_j(e)$. This yields:

$$v u^{-1} = s_i s_j \implies s_i^{-1} v = s_j u = y.$$

Since $s_i^{-1} \in \Gamma_i(e)$ and $s_j \in \Gamma_j(e)$, it follows that $y \in \Gamma_i(v) \cap \Gamma_j(u)$. This construction guarantees that for every vertex in the first intersection, there is a corresponding vertex in the second. Thus, $|\Gamma_i(u) \cap \Gamma_j(v)| \leq |\Gamma_i(v) \cap \Gamma_j(u)|$. By a symmetric argument, the reverse inequality holds, proving that $A_i A_j = A_j A_i$.

Conversely, suppose $A_i A_j = A_j A_i$. Let $r_i \in \Gamma_i(e)$ and $r_j \in \Gamma_j(e)$; we must show that $r_i r_j \in \Gamma_j(e) \Gamma_i(e)$. Consider the element $u = r_j \in \Gamma_j(e)$. Note that:

$$r_i^{-1}(r_i r_j) = r_j = u.$$

This implies $u \in \Gamma_i(r_i r_j) \cap \Gamma_j(e)$, meaning $|\Gamma_i(r_i r_j) \cap \Gamma_j(e)| \geq 1$. By the commutative hypothesis, we must have $|\Gamma_j(r_i r_j) \cap \Gamma_i(e)| \geq 1$. Thus, there exists some vertex $v \in \Gamma_j(r_i r_j) \cap \Gamma_i(e)$. 

Because $v \in \Gamma_i(e)$, we can write $v = s_i$ for some $s_i \in \Gamma_i(e)$. Because $v \in \Gamma_j(r_i r_j)$, there exists $s_j \in \Gamma_j(e)$ such that $s_j v = r_i r_j$. Substituting $v$, we obtain:

$$r_i r_j = s_j s_i \in \Gamma_j(e) \Gamma_i(e).$$

This shows $\Gamma_i(e) \Gamma_j(e) \subseteq \Gamma_j(e) \Gamma_i(e)$. By symmetric reasoning, the sets are equal, completing the proof.
\end{proof}

\begin{Corollary}
    Let $G = \operatorname{Cay}(H,S)$ be a Cayley graph. Then, $\Gamma$ is $i$-CDDR if and only if $\Gamma_i(e)\Gamma_j(e) = \Gamma_j(e)\Gamma_i(e)$ for every $0 \leq j \leq \operatorname{diam}(\Gamma)$. Analogously, $\Gamma$ is CDDR if and only if $\Gamma_i(e)\Gamma_j(e) = \Gamma_j(e)\Gamma_i(e)$ for every pair of distances $0 \leq i, j \leq \operatorname{diam}(G)$.
\end{Corollary}

\begin{Corollary}
    Let $G=\operatorname{Cay}(H,S)$. If $H$ is Abelian then $G$ is CDDR.
\end{Corollary}

\begin{theorem}
    Let $G=\operatorname{Cay}(H,S)$. Then $G$ is DWR if and only if for every $1 \leq i \leq \operatorname{diam}(G)$ and $k \geq 0$,
    $$S^k \cap \Gamma_i(e) = \Gamma_i(e) \quad \text{or} \quad S^k \cap \Gamma_i(e) = \emptyset.$$
\end{theorem}

\begin{proof}
    By the vertex-transitivity of Cayley graphs, analyzing walks between arbitrary vertices $u, v$ at distance $i$ is equivalent to analyzing walks between the identity $e$ and $vu^{-1} \in \Gamma_i(e)$. The condition that a walk of length $k$ exists from $e$ to $x$ is algebraically equivalent to $x \in S^k$.
    
    If $G$ is DWR, the existence of a length-$k$ walk depends exclusively on the distance $i$. Thus, if $S^k$ contains at least one element of $\Gamma_i(e)$, it must contain all of them, yielding $S^k \cap \Gamma_i(e) = \Gamma_i(e)$. Otherwise, it contains none, yielding $\emptyset$. 
    
    Conversely, the provided intersection condition forces a binary outcome: either all elements in $\Gamma_i(e)$ belong to $S^k$, or none do. Translating this back to arbitrary vertex pairs, the existence of a walk of length $k$ depends strictly on their distance $i$, proving $G$ is DWR.
\end{proof}

\begin{center}
   \begin{figure}[h]
\begin{center}
\begin{tikzpicture}[scale=0.8]
    \foreach \i in {1,...,6} {
        \coordinate (v\i) at ({180-\i*60}:3);
    }
    
    \foreach \punto in {1,...,6} {
        \fill (v\punto) circle(3pt);
       
    }
        
    
    \draw[thick] (v1) -- (v3);
    \draw[thick] (v2) -- (v4);
    \draw[thick] (v3) -- (v5);
    \draw[thick] (v4) -- (v6);
    \draw[thick] (v5) -- (v1);
    \draw[thick] (v6) -- (v2);
    
    \draw[thick] (v1) -- (v4);
    \draw[thick] (v2) -- (v5);
    \draw[thick] (v3) -- (v6);

\end{tikzpicture}
\end{center}
\caption{A Cayley Graph on Abelian group $G=(\mathbb{Z}_6,\{2,3,4\})$ which is not DWR }\label{fig: complemento_C6}
\end{figure}
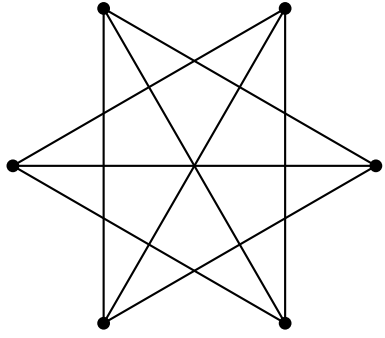

\end{center}

The preceding results characterize commutativity in terms of the layer set $\Gamma_i(e)$. A complementary sufficient condition, based on the symmetry of the generating set under conjugation, is given by the following.
A fundamental property of this construction is that the group $H$ acts on itself via right translations. For each $g \in H$, the mapping $\rho_g(x) = xg$ is a graph automorphism of $G$.Since right translations act as distance-preserving automorphisms, if $h \in H$ is at distance $i$ from the identity $e$, then $d(x, hx) = i$ for all $x \in H$. This ensures that the distance matrix $A_i$ can be decomposed as a sum of permutation matrices $P_h$ corresponding to the left regular representation of the group:$$A_i = \sum_{h \in \Gamma_i(e)} P_h$$

\begin{lemma}
Let $G = \operatorname{Cay}(H, S)$. If $h S h^{-1} = S$ for every $h \in \Gamma_i(e)$, then $G$ is $i$-CDDR.
\end{lemma}
\begin{proof}
Let $h \in \Gamma_i(e)$. The hypothesis $h S h^{-1} = S$ implies that the left translation $\sigma_h(x) = hx$ is also a graph automorphism of $G$. Since automorphisms preserve distances, we have $d(x, y) = d(hx, hy)$ for all $x, y \in V(G)$. In terms of permutation matrices, this structural symmetry is expressed as: $$P_h A_j P_h^T = A_j \quad \text{for every } 1 \leq j \leq \operatorname{diam}(G).$$
Since $P_h$ is an orthogonal matrix ($P_h^T = P_h^{-1}$), we obtain the direct commutativity relation $P_h A_j = A_j P_h$. Recalling that the distance matrix $A_i$ is the sum of the permutation matrices associated with the elements in $\Gamma_i(e)$, it linearly follows that: $$A_i A_j = \left( \sum_{h \in \Gamma_i(e)} P_h \right) A_j = \sum_{h \in \Gamma_i(e)} (P_h A_j) = \sum_{h \in \Gamma_i(e)} (A_j P_h) = A_j A_i.$$
This holds for all $1 \leq j \leq \operatorname{diam}(G)$, confirming that $G$ is $i$-CDDR.
\end{proof}

\begin{theorem}
Let $G = \operatorname{Cay}(H, S)$. If $h S h^{-1} = S$ for every $h \in S$, then $G$ is CDDR.
\end{theorem}
\begin{proof}

Let $N_H(S) = \{ g \in H \mid gSg^{-1} = S \}$ be the normalizer of the generating set $S$ in $H$. By hypothesis, $S \subseteq N_H(S)$. Since $N_H(S)$ is a subgroup of $H$ and $S$ generates $H$, it inherently follows that $H = \langle S \rangle \subseteq N_H(S)$.

This implies that $gSg^{-1} = S$ for every element $g \in H$. In particular, this invariant condition holds for every $h \in \Gamma_i(e)$ across all $1 \leq i \leq \operatorname{diam}(G)$. Applying the preceding lemma to each distance shell $i$, we conclude that every distance matrix $A_i$ commutes with every distance matrix $A_j$. Therefore, $G$ is a CDDR graph.

\end{proof}

\begin{Example}
Consider the symmetric group $S_n$ and let $T$ be the generating set consisting of all transpositions in $S_n$. It is a foundational result in group theory that conjugacy classes in $S_n$ are strictly determined by their cycle structure. Consequently, the set of all transpositions $T$ constitutes a single conjugacy class. This structural invariance implies that $g T g^{-1} = T$ for all $g \in S_n$. By the preceding theorem, it immediately follows that the Cayley graph $\operatorname{Cay}(S_n, T)$ is a CDDR graph for any $n \geq 2$.
\end{Example}

As illustrated in Figure \ref{fig: Mobius ladder}, the Möbius ladder serves as a concrete example of a graph that is CDDR and DMR, but strictly not DP. The following theorem generalizes this behavior to an infinite family of circulant graphs.

\begin{theorem}
    The Cayley Graph $G=\operatorname{Cay}(\mathbb{Z}_{4k},\{1,2k,4k-1\})$ is CDDR and not DP for every odd $k\geq 3$.
\end{theorem}
\begin{proof}
    We only have to prove that $G$ is not DP. We know that
    \begin{align*}
    A_1(G) &= \operatorname{circ}(0, \underbrace{1}_{1},0, \dots, 0, \underbrace{1}_{2k}, 0, \dots, 0, \underbrace{1}_{4k-1})=\text{circ}(\mathbf{v_1}) \\
    A_2(G) &= \operatorname{circ}(0, 0,\underbrace{1}_{2}, \dots, \underbrace{1}_{2k-1}, 0, \underbrace{1}_{2k+1}, \dots,  \underbrace{1}_{4k-2},0)=\text{circ}(\mathbf{v_2})
    \end{align*}
    And its well know that 
    \begin{align*}
        \mathbf{u_1}&=(1,-1,1,-1,\dots,1,-1,1,-1) \\
        \mathbf{u_2}&=(1,i,-1,-i,\dots,1,i,-1,-i)
    \end{align*}
    Are eigenvectors of $A_1(G)$ and $A_2(G)$. To know the eigenvalues we only have to make the inner products $\mathbf{v_i}.\mathbf{u}_j$. Using that $\mathbf{u_1}$ have $1$ in even positions and $-1$ in odd positions.
    \begin{equation*}
        \mathbf{v_1}\cdot \mathbf{u_1}= -1+1-1=-1 \quad  \mathbf{v_2}\cdot \mathbf{u_1}=1-1-1+1=0.
    \end{equation*}
    So $A_1(G)\mathbf{u_1}=-\mathbf{u_1}$ and $A_1(G)\mathbf{u_2}=\mathbf{0}$. For the other side, since $k$ is odd $\mathbf{u_2}$ has $-1$ in the position $2k$ then 
    \begin{equation*}
        \mathbf{v_2}\cdot \mathbf{u_1}=i-1-i=-1 \quad \mathbf{v_2}\cdot \mathbf{u_2}=-1+i-i-1=-2. 
    \end{equation*}
    So $A_1(G)\mathbf{u_1}=-\mathbf{u_1}$ and $A_2(G)\mathbf{u_2}=-2\mathbf{u_2}$. It follows for Proposition \ref{prop: caracterizacion de DP} that $G$ is not DP.

\end{proof}


\section{Conclusion}

In this work, we systematically studied how various distance-regularity classes are preserved under fundamental graph operations, specifically the Cartesian, Strong, Direct and Lexicographic products. Driven by the specific structural requirements of the direct product, we established a novel class of graphs—Distance Walk-Regular (DWR) graphs—and demonstrated their theoretical significance within the broader metric hierarchy. Furthermore, we have successfully characterized the class of Distance-Regular (DR) graphs by establishing an explicit condition on the distribution of walks between vertex pairs at fixed distances.

Furthermore, we introduced and characterized the family of $i$-CDDR graphs, offering a granular perspective on distance commutativity by relaxing global requirements. We also extended this structural analysis to the realm of group theory through Cayley graphs, providing explicit criteria for distance commutativity based on the algebraic interactions of distance-generating sets and their normalizers.

By leveraging these theoretical tools and product operations as lifting operators, we successfully constructed infinite examples of graphs that explicitly populate every possible intersection and symmetric difference of the studied regularity classes. To conclude, we present a graphical comparison of the different families of graphs studied in this work (see Fig \ref{fig: comparacion clases}). As a direct consequence of our constructions, no region in this structural hierarchy is empty; indeed, any specific combination of satisfying or strictly failing these regularity properties is instantiated by an infinite family of graphs.

\begin{figure}

\begin{center}
    \begin{tikzpicture}[scale=0.75, transform shape]
        \draw[thick] (-0.5,0) ellipse (1.5cm and 0.75cm);
        \node[] at (-0.5,0) {\textbf{DR}};
        \draw[thick] (1,0) ellipse (4.5cm and 2.25cm);
        \node[] at (4.5,0) {\textbf{CDDR}};
        %
        %
        \draw[thick] (0.5,0) ellipse (3cm and 1.5cm);
        \node[] at (2.5,0) {\textbf{DP}};
        %
        %
        %
        \draw[thick] (-2.5,0) ellipse (4cm and 1.75cm);
        \node[] at (-5.3,0) {\textbf{DMR}};
        %
        %
        %
        \draw[thick] (0,0) ellipse (8cm and 5cm);
        \node[] at (0,3) {\textbf{DDR}};
        %

        \draw[thick] (-0.5,-2.5) ellipse (3cm and 3.5cm);
        \node[] at (-0.5,-5.4) {\textbf{DWR}};

        \draw[thick] (1.5,0) ellipse (6cm and 2.5cm);
        \node[] at (6.4,0) {\textbf{1-CDDR}};
        
    \end{tikzpicture}
\end{center}
\caption{The inclusion relationships and structural hierarchy among the families}\label{fig: comparacion clases}
\end{figure}
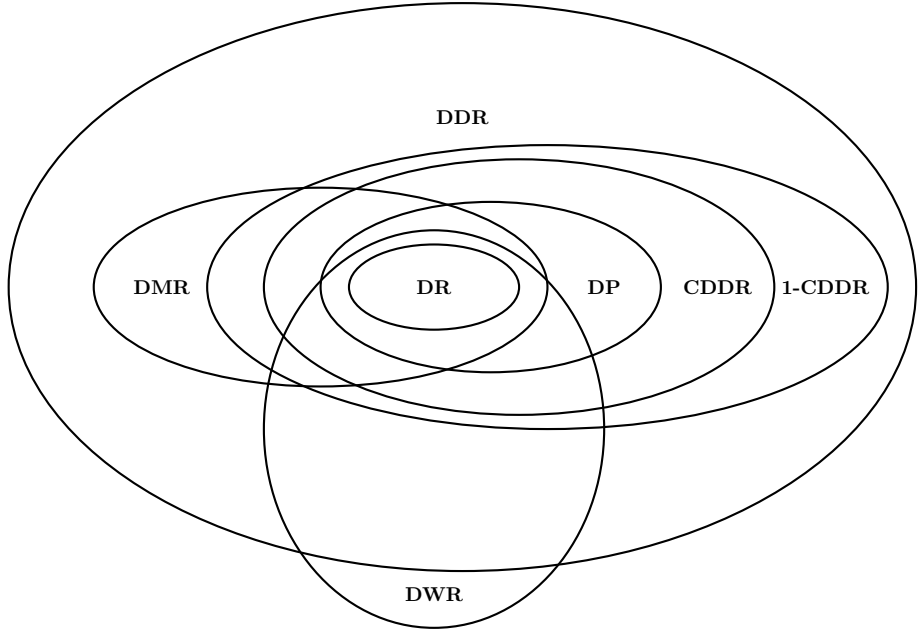

\vspace{0.3cm}
\noindent\textbf{Open questions.}
This work leaves open several compelling avenues for future research, most notably regarding the interplay between group structure and commutativity. We conjecture that for Cayley graphs over solvable groups, the $1$-CDDR property might strictly imply CDDR, suggesting a structural rigidity imposed by the group's subnormal series. Furthermore, characterizing DWR graphs through purely spectral conditions remains an elusive challenge, as does the identification of minimal structural obstructions that prevent graph products from inheriting $i$-CDDR properties. We hope these questions provide a fruitful path for further exploring the boundaries of distance-regularity

\bibliographystyle{plain}

\end{document}